\documentclass{amsart}
\usepackage{amsmath, amssymb, amsthm}
\usepackage{pgfplots}
\usepackage{tikz-cd}
\usepackage{url}
\usepackage[all]{xy}

\title{A divisibility result on combinatorics of generalized braids}
\author{Loic Foissy and Jean Fromentin}

\theoremstyle{plain}
\newtheorem{prop}{Proposition}[section]
\newtheorem{thrm}{Theorem}[section]
\theoremstyle{definition}
\newtheorem{defi}[prop]{Definition}
\newtheorem{exam}[prop]{Example}
\newtheorem{nota}[prop]{Notation}
\newtheorem{lem}[prop]{Lemma}
\newtheorem{coro}[prop]{Corollary}

\DeclareMathOperator\prodC{prod}
\DeclareMathOperator\Gar{Gar}
\DeclareMathOperator\card{card}
\DeclareMathOperator\Adj{Adj}
\DeclareMathOperator\FQSym{{\bf FQSym}}
\DeclareMathOperator\BFQSym{{\bf BFQSym}}

\DeclareMathOperator\Ima{Im}
\DeclareMathOperator\len{\ell}

\DeclareMathOperator\std{Std}
\DeclareMathOperator\perm{\rho}
\DeclareMathOperator\eps{\varepsilon}

\newcommand\del[1]{\text{del}_{#1}}
\newcommand\dd{d}
\newcommand\ddp{{d+1}}
\newcommand\ddo{{d-1}}

\newcommand\Des[1]{\text{Des}\left(#1\right)}
\renewcommand\geq{\geqslant}
\newcommand{\kk}{k}
\newcommand{\kkp}{{k+1}}
\newcommand{\kko}{{k-1}}
\newcommand\ii{i}
\newcommand\iio{{i-1}}
\newcommand\iip{{i+1}}
\newcommand\jj{j}
\newcommand\jjo{{j-1}}
\newcommand\jjp{{j+1}}
\newcommand\dec[2]{#1\hspace{-2pt}\left<#2\right>}

\newcommand\ie{i.e.}
\newcommand\inv{{}^{-1}}
\renewcommand\leq{\leqslant}
\newcommand\gcdl{\wedge}
\newcommand\BPB[1]{B\hspace{-0.1em}B_{#1}^+}
\newcommand\nb[2]{b_{#1,#2}}
\newcommand\NN{\mathbb{N}}
\newcommand\nn{n}
\newcommand\nno{{n-1}}
\newcommand\SPB[1]{S\hspace{-0.1em}B_{#1}}
\newcommand\nnp{{n+1}}

\newcommand\shift[2]{#1{\scriptstyle [#2]}}
\newcommand\sign[1]{\textrm{sign}_{#1}}
\newcommand\Sym[1]{\mathfrak{S}_{#1}}
\newcommand\SSym[1]{\mathfrak{S}_{#1}^{\pm}}
\def\shuff#1#2{\mathbin{
      \hbox{\vbox{\hbox{\vrule \hskip#2 \vrule height#1 width 0pt}\hrule}\vbox{\hbox{\vrule \hskip#2 \vrule height#1 width 0pt\vrule }\hrule}}}}\newcommand\QQ{\mathbb{Q}}
\def\shuffl{{\mathchoice{\shuff{5pt}{3.5pt}}{\shuff{5pt}{3.5pt}}{\shuff{3pt}{2.6pt}}{\shuff{3pt}{2.6pt}}}}

\def\shuffle{\, \shuffl \,}
\newcommand\ZZ{\mathbb{Z}}
\newcommand\SH[2]{\text{Sh}_{#1,#2}}
\newcommand\WW[1]{\text{W}^\pm_{#1}}
\newcommand{\YY}{Y}
\newcommand\uu{u}
\newcommand\vv{v}
\newcommand\ww{w}
\subjclass[2000]{20F36, 05A05, 16T30}

\begin{document}

\begin{abstract}
For every finite Coxeter group $\Gamma$, each positive braids in the corresponding braid group
admits a unique decomposition as a finite sequence of elements of $\Gamma$, the so-called Garside-normal form.
The study of the associated adjacency matrix $\Adj(\Gamma)$ allows to count the number of Garside-normal form of a given length.
In this paper we prove that the characteristic polynomial of~$\Adj(B_n)$ divides the one of~$\Adj(B_\nnp)$.
The key point is the use of a Hopf algebra based on signed permutations.
A similar result was already known for the type $A$. 
We observe that this does not hold for type $D$.
The other Coxeter  types  ($I$, $E$, $F$ and $H$)  are also studied.
\end{abstract}

\maketitle

\section*{Introduction}

Let $S$ be a set. A \emph{Coxeter matrix} on $S$ is a symmetric matrix  $M=(m_{s,t})$ whose entries are in $\NN\cup\{+\infty\}$ 
and such that $m_{s,t}=1$ if, and only if, $s=t$. 
A Coxeter matrix is usually represented by a labelled \emph{Coxeter graph} $\Gamma$ whose vertices are the elements of $S$;
there is an edge between $s$ and $t$ labelled $m_{s,t}$ if, and only if, $m_{s,t}\geq 3$.
From such a graph $\Gamma$,  we define a group $W_\Gamma$ by the presentation
\begin{equation*}
W_\Gamma=\left<S\, \left| \begin{array}{cl} 
			s^2=1&\text{for $s\in S$}\\
			\prodC(s,t;m_{s,t})=\prodC(t,s;m_{t,s}) &\text{for $s,t\in S$ and $m_{s,t}\not=+\infty$}
		   \end{array}\right.\right>.
\end{equation*}
where $\prodC(s,t;m_{s,t})$ is the product $s\,t\,s...$ with $m_{s,t}$ terms. 
The pair $(W_\Gamma,S)$ is called a \emph{Coxeter system}, and $W_\Gamma$ is the \emph{Coxeter group} of type $\Gamma$.
Note that two elements $s$ and $t$ of $S$ commute in $W_\Gamma$ if, and only if, $s$ and $t$ are not connected in~$\Gamma$.
Denoting  by $\Gamma_1,...,\Gamma_k$ the connected components of $\Gamma$, we obtain that $W_\Gamma$ is the direct product $W_{\Gamma_1}\times...\times W_{\Gamma_k}$. 
The Coxeter group $W_\Gamma$ is said to be \emph{irreducible} if the Coxeter graph $\Gamma$ is connected. 
We say that a Coxeter graph is \emph{spherical} if the corresponding group $W_\Gamma$ is finite.
There are four infinite families of connected spherical Coxeter graph:  $A_n$ $(n\geq 1)$, $B_n$ $(n\geq 2)$, $D_n$ $(n\geq 4)$, $I_2(p)$ ($p\geq 5$),
 and six exceptional graphs $E_6$, $E_7$, $E_8$, $F_4$, $H_3$ and $H_4$.
For $\Gamma=A_n$, the group~$W_\Gamma$ is the symmetric group $\Sym\nnp$.

For a Coxeter graph~$\Gamma$, we define the \emph{braid group} $B(W_\Gamma)$ by the presentation 
\begin{equation*}
B(W_\Gamma)=\big<S\, \big|\,  \prodC(s,t;m_{s,t})=\prodC(t,s;m_{t,s}) \ \text{for $s,t\in S$ and $m_{s,t}\not=+\infty$} \big>.
\end{equation*}
and the positive braid monoid to be the monoid presented by
\begin{equation*}
B^+(W_\Gamma)=\big<S\, \big|\,  \prodC(s,t;m_{s,t})=\prodC(t,s;m_{t,s}) \ \text{for $s,t\in S$ and $m_{s,t}\not=+\infty$} \big>^+.
\end{equation*}
The groups $B(W_\Gamma)$ are known as Artin-Tits groups; they have been introduced 
in~\cite{deligne1972Lesimmdesgrodetregen,brieskorn1972ArtundCox} and in \cite{humphreys1992RefgroCoxgro} for spherical type.
The embedding of the monoid $B^+(W_\Gamma)$ in the corresponding group $B(W_\Gamma)$ was established by L.~Paris in \cite{paris2002Artmoninjthegro}.
For $\Gamma=A_n$, the braid group $B(W_{A_n})$ is the Artin braid group $B_n$ and $B^+(W_{A_n})$ is the monoid of positive Artin braids $B_n^+$.

We now suppose that $\Gamma$ is a spherical Coxeter graph.
The Garside normal form allows us to express each braid $\beta$ of $B^+(W_\Gamma)$ as a unique finite sequence of elements of $W_\Gamma$.
This defines an injection $\Gar$ form $B^+(W_\Gamma)$ to $W_\Gamma^{(\NN)}$.
The Garside length of a braid $\beta\in B^+(W_\Gamma)$ is the length of the finite sequence~$\Gar(\beta)$.
If, for all $\ell\in\NN$, we denote by $B^\ell(W_\Gamma)$ the set of braids whose Garside length is $\ell$, 
the map $\Gar$ defines a bijection between $B^\ell(W_\Gamma)$ and $\Gar(B^+(W_\Gamma))\cap W_\Gamma^\ell$.

A sequence $(s,t) \in W_\Gamma^2$ is said \emph{normal} if $(s,t)$ belongs to $B^2(W_\Gamma)$. 
From the local characterization of the Garside normal form, for $\ell\geq 2$ the sequence $(w_1,...,w_\ell)$ of $W_\Gamma^\ell$ belongs to $\Gar(B^+(W_\Gamma))$ if, and only if, $(w_i,w_\iip)$ is normal for all $i=1,...,\ell-1$.
Roughly speaking, in order to recognize the elements of $\Gar(B^+(W_\Gamma))$ among thus of $W_\Gamma^{(\NN)}$ it is enough to recognize the elements of  $B^2(W_\Gamma)$ among thus of $W_\Gamma^2$.

We define a square matrix $\Adj_\Gamma=(a_{u,v})$ indexed by the elements of $W_\Gamma$ by
\begin{equation*}
a_{u,v}=\begin{cases}
1 & \text{if $(u,v)$ is normal},\\
0 & \text{otherwise}.
\end{cases}
\end{equation*} 
For $\ell\geqslant 1$, the number of positive braids whose Garside length is $\ell$ is then
\begin{equation*}
\card(B^{\ell}(W_\Gamma))={}^t X \Adj_\Gamma^{\ell-1} X,\qquad\text{where}\  X_u=\begin{cases} 0 & \text{if $u=1_{W_\Gamma}$} \\ 1 & \text{otherwise}\end{cases}.
\end{equation*}
Therefore the eigenvalues of $\Adj_\Gamma$ give informations on the growth of $\card(B^{\ell}(W_\Gamma))$ relatively to $\ell$.

Assume that $\Gamma$ is a connected spherical type graph of one of the infinite family $A_n, B_n$ or $D_n$.
We define $\chi_n^A$, $\chi_n^B$ and $\chi_n^D$ to be the characteristic polynomials of $\Adj_{A_n}, \Adj_{B_n}$ and $\Adj_{D_n}$ respectively. 
In \cite{dehornoy2007Comnorseqbra}, P.~Dehornoy conjectures  that $\chi_n^A$ is a divisor of $\chi_\nnp^A$. 
This conjecture was proved by F. Hivert, J.C. Novelli and J.Y. Thibon in \cite{hivert2008SurunecondeDeh}.
To prove that $\chi_n^A$ divides $\chi_\nnp^{A}$, they see 
$\Adj_{A_n}$ as the matrix of an endomorphism $\Phi^A_n$ of the Malvenuto-Reutenauer Hopf algebra~$\FQSym$ 
\cite{MalvenutoReutenauer,DuchampHivertThibon}. 
We recall that $\FQSym$ is a connected graded Hopf algebra whose a basis in degree~$n$ is indexed by the element of $\Sym\nn\simeq W_{A_\nno}$. 
 The authors of \cite{hivert2008SurunecondeDeh} then construct a surjective derivation $\partial$ of degree $-1$ satisfying 
$\partial\circ\Phi^A_n=\Phi^A_\nno\circ \partial$, and eventually prove the divisibility result.
A combinatorial description of $\Adj_{A_n}$ can be found in~\cite{dehornoy2007Comnorseqbra} and in 
\cite{gebhardt2012Couverbipgracomgrofunbramon}, with a more algorithmic approach.

The aim of this paper is to prove that the polynomial $\chi_n^B$ divides the polynomial~$\chi_{n+1}^B$. 
The first step is to construct a Hopf algebra $\BFQSym$ from $W_{B_n}$ which plays the same role for the type $B$ as $\FQSym$ for the type $A$;
this is a special case of a general construction for families of wreath products, see \cite{novelli2010Frequafundesalgwreprononmulfun}.
We then interpret  $\Adj_{B_n}$ as the matrix of an endomorphism $\Phi^B_n$ of the Hopf algebra~$\BFQSym$. 
The next step is to construct a derivation $\partial$ on $\BFQSym$ satisfying the relation $\partial\circ\Phi^B_n=\Phi^B_\nno\circ \partial$ and establish the divisibility result. 
Unfortunately there is no such a result for the Coxeter type $D_n$: the polynomial $\chi_{4}^D$ is not a divisor of $\chi_{5}^D$ neither of $\chi_{6}^D$.

The paper is divided as follows. 
The first section is an introduction to Coxeter groups and braid monoids of type $B$.
The adjacency matrix $\Adj_{B_n}$ is introduced here.
Section 2 is devoted to the Hopf algebra $\BFQSym$.
In Section 3, we prove the divisibility result using a derivation on the Hopf algebra $\BFQSym$.
Conclusions and characteristic polynomials of type $D$, $I$, $E$, $F$ and $H$ are in the last section.

\section{Coxeter groups and braid groups of type $B$.}

The following notational convention will be useful in the sequel:
if $p\leq q$ in $\mathbb{Z}$, we denote by $[p,q]$ the subset $\{p,...,q\}$ of $\mathbb{Z}$.

\subsection{Signed permutation groups}

\label{S:Signed}

\begin{defi}
A \emph{signed permutation} of rank $n$ is a permutation $\sigma$ of $[-n,n]$ satisfying $\sigma(-i)=-\sigma(i)$ for all $i\in [-n,n]$.
We denote by $\SSym\nn$ the group of signed permutations.
\end{defi}

In the literature, the group of signed permutations $\SSym\nn$ is also known as the hyperoctahedral group of rank $n$. 
We note that, by very definition, all signed permutations send $0$ to itself. 
Always by definition, a signed permutation is entirely defined by its values on $[1,n]$. 
In the sequel, a signed permutation $\sigma$ of rank $n$ will consequently be written as $(\sigma(1), ..., \sigma(n))$. 
This notation is often called the \emph{window} notation of the permutation~$\sigma$.

\begin{defi}
 For $\sigma$ a signed permutation of $\SSym\nn$, the \emph{word} of $\sigma$,
 denoted by $w(\sigma)$ is the word $\sigma(1)\,...\,\sigma(\nn)$ on the alphabet $[-n,n]\setminus\{0\}$.
\end{defi}

\begin{exam}
 Signed permutations of rank $2$ are 
\begin{equation*}
\SSym2=\{(1,2),(-1,2),(1,-2),(-1,-2),(2,1),(-2,1),(2,-1),(-2,-1)\}.
\end{equation*}
\end{exam}

One remarks that for any signed permutation $\sigma$ of $\SSym\nn$,
the map $|\sigma|$ defined on~$[1,\nn]$ by $|\sigma|(i)=|\sigma(i)|$
is a permutation of $\Sym\nn$.

Among the signed permutations, we isolate a generating family $s_i$'s which eventually equips 
$\SSym\nn$ with a Coxeter structure.

\begin{defi}
Let $\nn\geq 1$. We define a permutation $s_i^{(n)}$ of $\SSym\nn$ by $s_0^{(n)}=(-1,2,...,n)$ and $s_i^{(n)}=(1,...,i+1,i,...,n)$
 for $i\in[1,n]$.
 \end{defi}
 
 From the natural injection of $\SSym\nn$ to $\SSym\nnp$ we can write $s_i$ instead of $s_i^{(n)}$ without ambiguity. The following proposition is a direct consequence of the previous definition.
 
 \begin{prop}
For all $n\geq 1$, the permutations $S_n=\{s_0,...,s_n\}$  are subject to the relations:

-- $R_1(S_n)$: $s_i^2=1$ for all $i\in[0,n]$;

-- $R_2(S_n)$: $s_0\,s_1\,s_0\,s_1=s_1\,s_0\,s_1\,s_0$;

-- $R_3(S_n)$: $s_i\,s_j=s_j\,s_i$ for $i,j\in[0,n]$ with $|i-j|\geq 2$;

-- $R_4(S_n)$: $s_i\,s_j\,s_i=s_j\,s_i\,s_j$ for $1\leq i,j\leq n$ with $|i-j|=1$.

\end{prop}

Each signed permutation $\sigma$ of $\SSym\nn$ can be represented as a product of the $s_i$'s.
Some of these representations are shorter than the others.
The minimal numbers of $s_i$'s required is then a parameter of the signed permutation.

\begin{defi}
Let $\sigma$ a signed permutation of $\SSym\nn$. 
The \emph{length} of $\sigma$ denoted by~$\len(\sigma)$ is the minimal integer $k$
such that there exists $x_1,...,x_k$ in $S_n$ satisfying $\sigma=x_1\cdot...\cdot x_k$. 
An expression of $\sigma$ in terms of $S_n$ is said to be \emph{reduced} if it has length $\len(\sigma)$.
\end{defi}

\begin{exam}
\label{E:SSym3}
Permutations of $\SSym3$ admit the following decompositions in terms of permutations in $s_i$'s:
\[
\begin{array}{rclcrcl}
(1,2)   & = & \emptyset                      &\hspace{2em}& (2,1)   & = & s_1\\
(-1,2)  & = & s_0                            &            & (-2,1)  & = & s_1\cdot s_0\\
(1,-2)  & = & s_1\cdot s_0\cdot s_1          &            & (2,-1)  & = & s_0\cdot s_1\\
(-1,-2) & = & s_0\cdot s_1\cdot s_0 \cdot s_1&            & (-2,-1) & = & s_0 \cdot s_1 \cdot s_0
 \end{array}
\]
Each given expression is reduced. In particular, the length of $(-1,-2)$ is $4$, while the length of $(-2,1)$ is $2$.
\end{exam}

Among all the signed permutations of $\SSym\nn$, there is a unique one with maximal length called  \emph{Coxeter element} of $\SSym\nn$ and denoted by $w_\nn^B$: 
\begin{equation*}
 w_\nn^B=(-1, ... , -\nn).
\end{equation*}

A presentation of $\SSym\nn$ is given by  relations $R_1,R_2,R_3$ and $R_4$ on $S_\nn$.
More precisely the group of signed permutations $\SSym\nn$ is isomorphic to the Coxeter group~$W_{B_n}$ with generator set $S_n$ and relations given by the following graph:
\[
\begin{tikzpicture}
\draw(0,0) node {$B_n:$};
\draw(1.1,0) node[below]{\small$4$};
\draw(2.1,0) node[below]{\small$3$};
\draw(3.1,0) node[below]{\small$3$};
\draw(5.1,0) node[below]{\small$3$};
\draw(0.6,0) -- (3.6,0);
\draw(3.6,0) -- (4.6,0) [dotted];
\draw(4.6,0) -- (5.6,0);
\filldraw(0.6,0) circle (2pt) node[above]{\small$s_0$};
\filldraw(1.6,0) circle (2pt) node[above]{\small$s_1$};
\filldraw(2.6,0) circle (2pt) node[above]{\small$s_2$};
\filldraw(3.6,0) circle (2pt) node[above]{\small$s_3$};
\filldraw(4.6,0) circle (2pt) node[above]{\small$s_{n-2}$};
\filldraw(5.6,0) circle (2pt) node[above]{\small$s_{n-1}$};
\end{tikzpicture}
\]

For more details the reader can consult \cite{bjorner2006ComCoxgro}.
Thanks to this isomorphism, we identify the group~$\SSym\nn$ with $W_{B_\nn}$ for $n\geq1$.
%We then extend the family $(B_n)$ defined for $n\geq 2$ to include $B_1=A_1$.

\subsection{Braid monoids of type $B$.}

Putting $\Theta_n^B=\{\theta_0,...,\theta_\nno\}$, the braid monoid of type $B$ and of rank~$n$ is the monoid $\BPB\nn$ whose presentation is
 \begin{equation*}
\BPB\nn=B^+\left(\SSym\nn\right)=B^+\left(W_{B_n}\right)=\left<\Theta_n^B\,|\,R_2\left(\Theta_n^B\right),\ R_3\left(\Theta_n^B\right)\ \text{and}\ R_4\left(\Theta_n^B\right)\right>^+.
 \end{equation*}

% Relations $R_2$, $R_3$ and $R_4$ preserve the length of the involving word.
% Hence the length of a representing word (in the letters $\Theta_n^B)$ of a braid 
% $x$ of $\BPB\nn$ depends only on $x$. 
% We then define the length of a braid in $\BPB\nn$ to be the length of one 
% of its representating word.

The group of signed permutations $\SSym\nn$ is a quotient of $\BPB\nn$ by $\theta_\ii^2=1$.
We denote by $\pi$ the natural surjective homomorphism defined by:
\[
\begin{array}{rcl}
\pi:\BPB\nn&\to&\SSym\nn\\
\theta_\ii&\mapsto&s_\ii.
\end{array}
\]

The following result is fundamental in the study of Coxeter groups, and is known as the \emph{Exchange Lemma}.

\begin{lem}[Theorem 1.4.3 of \cite{bjorner2006ComCoxgro}]
Let $x_1\,...\,x_\kk$ be a reduced expression of a signed permutation $\sigma\in\SSym\nn$ and $\ii\in[0,...,\nno]$.
If $\ell(\sigma\,s_i)<\ell(\sigma)$, then there exists $\jj\in[1,\kk]$ such that $\sigma s_i$ is equal to 
$x_1\,...\,\widehat{x_\jj}\,...\,x_\kk$.
\end{lem}

A consequence of the \emph{Exchange Lemma} is that we can go from a reduced expression 
of a signed permutation to another only by applying relations of type $R_2$, $R_3$ and~$R_4$;
in other words, relation $s_i^2=1$ can be avoided, see \cite{Digne} for more details.

\begin{defi}
For $\sigma$ in $\SSym\nn$ we define $r(\sigma)$ to be the braid $\theta_{i_1}\,...\,\theta_{i_\kk}$ where $s_{i_1}\,...\,s_{i_\kk}$
is a reduced expression of $\sigma$.
\end{defi}

Since relations $R_2$, $R_3$ and $R_4$ are also verified by the $\theta_\ii$'s, the braid $r(\sigma)$ 
is well defined for every signed permutation $\sigma$.

\begin{prop}
\label{P:RInj}
 For $\nn\geq0$, the map $r:\SSym\nn\to\BPB\nn$ is injective.
\end{prop}

This is a direct consequence of the definition of $r$.

\begin{defi}
 A braid $x$ of $\BPB\nn$ is \emph{simple} if it belongs to $r\left(\SSym\nn\right)$. 
 We denote by~$\SPB\nn$ the set of all simple braids.
 The element $\Delta_n^B=r(w_n^B)$ is the \emph{Garside} element of $\BPB\nn$.
\end{defi}

In particular, there are $2^\nn\nn!$ simple braids in $\BPB\nn$.
Simple braids are used to describe the structure of the braid monoid $\BPB\nn$ 
from the one of the Coxeter group~$\SSym\nn\simeq W_{B_n}$.

\begin{exam}
Using Example~\ref{E:SSym3} we have that the simple braids of $\BPB2$ are
 \[  
  \SPB2=\{1,\theta_0,\theta_1,\theta_0\theta_1,\theta_1\theta_0,\theta_1\theta_0\theta_1,\theta_1\theta_0\theta_1,\theta_0\theta_1\theta_0\theta_1\}
 \]
The Coxeter element of $\SPB2$ is $w_2^B=(-1,-2)$, whose 
a decomposition in terms of the $s_i$'s is $w_2^B=s_0\,s_1\,s_0\,s_1$ and so $ \Delta_2^B=\theta_0\,\theta_1\,\theta_0\,\theta_1$.
\end{exam}

\begin{defi}
 Let $x$ and $y$ be two braids of $\BPB\nn$. 
 We say that $x$ left divides $y$ or that $y$ is a right multiple of $x$ if there exists 
 $z\in \BPB\nn$ satisfying $x.z=y$. 
 %We denote by $\DivL(x)$ the set of all the right divisiors of $x$.
\end{defi}

The Coxeter group $\SSym\nn$ is equipped with a lattice structure via the relation $\preccurlyeq$ defined 
by $\sigma\preccurlyeq\tau$ iff $\len(\tau)=\len(\sigma)+\len(\sigma\inv\tau)$.
Equipped with the left divisibility, the set $\SPB\nn$ is a lattice which is isomorphic to $(\SSym\nn,\preccurlyeq)$.
The maximal element of $\SSym\nn$ is $w_\nn^B$, while the one of $\SPB\nn$ is $\Delta_\nn^B$.
There is also an ordering $\succcurlyeq$ on $\SSym\nn$ such that  $\SPB\nn$ equipped with the right divisibility 
is a lattice,  isomorphic to $(\SSym\nn,\succcurlyeq)$.
In particular, simple elements  of $\BPB\nn$ are exactly the left (or the right) divisors of $\Delta_\nn^B$.

\begin{nota}
 For $x$ and $y$ two braids of $\BPB\nn$, we denote by $x\gcdl y$ the left great common divisor of $x$ and $y$. 
\end{nota}

\subsection{Left Garside normal form}

Let $x$ be a non trivial braid of $\BPB\nn$.  
The left great common divisor $x_1$ of $x$ and $\Delta_n^B$ is a simple element. 
Since one of the braids $\theta_i$'s (which are simple) 
left divides $x$, the braid $x_1$ is non trivial. 
We can write $x$ as $x=x_1\cdot x'$, with $x'\in\BPB\nn$.
If the braid $x'$ is trivial, we are done; else, we restart the process, replacing $x$ by $x'$.
As the length of the involved braid strictly decrease, we eventually obtain the trivial braid.

\begin{prop}
\label{P:Gar}
Let $x\in\BPB\nn$ be a non trivial braid.
There exists a unique integer $k\geq 1$ and unique non trivial simple braids $x_1,...,x_k$ satisfying

$(i)$ $x=x_1\cdot...\cdot x_k$;

$(ii)$ $x_i=(x_i\cdot...\cdot x_k)\gcdl\Delta^B_n$ for $i\in[1,k-1]$.

\noindent The expression $x_1\cdot...\cdot x_k$ is called the \emph{left Garside normal form} of the braid $x$.
\end{prop} 

The proof of the previous Proposition is a classic Garside result and can be found in \cite{hivert2008SurunecondeDeh}.
Note that  in Proposition~\ref{P:Gar}, we exclude the trivial braid from the decomposition.
This is done in order to have unicity for the integer $k$.
Indeed, one can transform a decomposition $x=x_1\cdot ...\cdot x_k$ to
$x=x_1\cdot...\cdot x_k \cdot 1 \cdot ... \cdot 1$ that satisfy conditions $(i)$ and $(ii)$.
The price to pay is that the trivial braid must be treated separately.

\begin{defi}
The integer $k$ introduced in the previous proposition is the \emph{Garside length} of the braid $x$.
By convention the \emph{Garside length} of the trivial braid is~$0$, corresponding to the empty product of simple braids.
\end{defi}

\begin{exam}
\label{X:Gar}
 Let $x=\theta_1\theta_1\theta_0\theta_1\theta_0\theta_1$ be a braid of $\BPB2$. 
 The maximal prefix of the given expression of $x$ that is a word of a simple braid is $\theta_1$.
 However, using relation~$R_2$ on the underlined factor of $x$ we obtain:
 \[
  x=\theta_1\underline{\theta_1\theta_0\theta_1\theta_0}\theta_0=\theta_1\underline{\theta_0\theta_1\theta_0\theta_1}\theta_0
 \]
The braid $y=\theta_1\theta_0\theta_1\theta_0$ is then a left divisor of $x$.
As $y$ is equal to the simple braid $\Delta_2^B$, we have $x_1=y$ and then
$x=x_1\cdot\theta_1\theta_0$.
Since $y=\theta_1\,\theta_0$ is simple, we have $x_2=\theta_1\theta_0$.
We finally obtain 
\[
 x=x_1\cdot x_0=\theta_1\theta_0\theta_1\theta_0\cdot\theta_1\theta_0,
\]
establishing that the Garside length of the braid $x$ is $2$.
\end{exam}

Condition $(ii)$ of Proposition~\ref{P:Gar} is difficult to check in practice. 
However it can replaced by a local condition, involving only two consecutive terms of the left Garside normal form.
More precisely, $(ii)$ is equivalent to

$(ii')$  the pair $(x_i,x_\iip)$ is normal for $i\in [1,k-1]$.

\begin{defi}
 A pair $(x,y)\in\SPB\nn^2$ of simple braids is said to be \emph{normal} if the relation~$x=(x\cdot y)\gcdl \Delta_n^B$ holds.
 \end{defi}

Since the number of simple elements is finite, there is a finite number of braids 
with a given Garside length.

\begin{defi}
For positive integer $n$ and $d$, we denote by $\nb\nn\dd$ the number of braids of $\BPB\nn$ which are of Garside length $d$.
\end{defi}

In order to determine $\nb\nn\dd$, we will switch to the Coxeter context.

\subsection{Combinatorics of normal sequences}

We recall that each simple braid of~$\SPB\nn$ can be expressed as $r(\sigma)$, where
$\sigma$ is a signed permutation. 
From the definition of normal pair of braids, we obtain a notion of normal pair of signed permutations.
We say that a pair $(\sigma,\tau)$ of $\SSym\nn$ is normal if $(r(\sigma),r(\tau))$ is.
Thus Proposition~\ref{P:Gar} can be reformulated as follow:

\begin{prop}
\label{P:GarPerm}
 For $n\geq 2$ and $x\in\BPB\nn$ a non trivial braid, there exists a unique integer $k\geq 1$ 
 and non trivial signed permutations
 $\sigma_1,...,\sigma_k$ of $\SSym\nn$ satisfying the following relations:
 
 $(i)$ $x=r(\sigma_1)\cdot...\cdot r(\sigma_k)$;
 
 $(ii)$ the pair $(\sigma_i,\sigma_\iip)$ is normal for $i\in[1,k-1]$.
 
\end{prop}

Instead of counting braids of Garside length $\dd$, we will count sequences of signed permutations of length $\dd$ which are normal.

\begin{defi}
 A sequence $(\sigma_1,...,\sigma_k)$ of signed permutations is \emph{normal} if  the pair $(\sigma_i,\sigma_\iip)$ is normal for $i\in[1,k-1]$. 
\end{defi}

The number $\nb\nn\dd$ is then the number of length $\dd$ normal sequences of
non trivial signed permutations of~$\SSym\nn$.
We now look for a criterion for a pair to be normal in the Coxeter context.

\begin{defi}
\label{D:Des}
The \emph{descent set} of a permutation $\sigma\in\SSym\nn$ is defined by
\begin{equation*}
  \Des\sigma=\{i\in [0,\nno] \, | \, \len(\sigma\,s_i)<\len(\sigma)\}
 \end{equation*}
\end{defi}

\begin{exam}
 Let us compute the descent set of $\sigma=(-2,1)$. 
 A reduced expression of $\sigma$ is $s_1\,s_0$ and so $\sigma$ has length $2$.
 The expression $\sigma\,s_0=s_1\,s_0\,s_0$ reduce to $s_1$, which is of length $1$.
 The expression $\sigma\,s_1=s_1\,s_0\,s_1$ is reduced, and so $\sigma\,s_1$
 has length~$3$.
 Therefore the descents set of $\sigma$ is $\Des\sigma=\{0\}$.
 \end{exam}

 Let us start with two intermediate results.

 \begin{lem}
 \label{L:Des1}
Let $\sigma$ be a signed permutation of $\SSym\nn$, and $i\in[0,\nno]$.
The braid $r(\sigma)\theta_i$ is simple if, and only if, $\ii\not\in\Des\sigma$.
 \end{lem}
 
 \begin{proof}
 Let $\sigma$ be a signed permutation of $\SSym\nn$ and $t_1\,...\,t_{\ell(\sigma)}$ one of its reduced expression.
 If $\ii\not\in\Des\sigma$ then $\len(\sigma s_i)>\ell(\sigma)$ holds. 
 Hence $t_1\,...\,t_{\ell(\sigma)}\,s_i$ is a reduced expression of $\sigma s_i$.
 It follows  $r(\sigma s_i)=r(t_1\,...\,t_{\ell(\sigma)}) r(s_i)=r(\sigma) \theta_\ii$, and so $r(\sigma)\theta_\ii$ is simple.
 Conversely, let us assume that $r(\sigma)\theta_\ii$ is simple. 
 There exists a signed permutation  $\tau$ in~$\SSym\nn$ of length $\ell(\sigma)+1$ satisfying $\pi(r(\sigma)\theta_\ii)=\tau$.
 As $\pi(r(\sigma)\theta_\ii)$ is $\sigma s_i$, we must have $\len(\sigma s_i)=\len(\sigma)+1$ and so $i\not\in \Des\sigma$.
 \end{proof}

 \begin{lem}
 \label{L:Des2}
 For $\tau$ a signed permutation of $\SSym\nn$ and $i\in[0,\nno]$,
  the braids $\theta_i$ is a left divisor of $r(\tau)$ if, and only if, $\ii\in\Des{\tau\inv}$.
 \end{lem}

\begin{proof}
 The braids $\theta_\ii$ and $r(\tau)$ are simple.
 Thanks to the lattice isomorphism between $\SPB\nn$ equipped with the left divisibility and $(\SSym\nn,\preccurlyeq)$, 
 the braid $\theta_\ii$ is a left divisor of $r(\tau)$ if and only $s_i\preccurlyeq \tau$ holds, and so,
 by definition of $\preccurlyeq$ if, and only if,  $\len(\tau)=\len(s_i)+\len(s_i\tau)$,
 which is equivalent to $\len(s_i\tau)<\len(\tau)$. 
 As the length of a permutation is the length of its inverse,
we have $\len(s_i\tau)<\len(\tau)\Leftrightarrow \len(\tau\inv s_i)<\len(\tau\inv)$
which is equivalent to $\ii\in\Des{\tau\inv}$.
\end{proof}

\begin{prop}
\label{P:NorDes}
 A pair $(\sigma,\tau)$ of signed permutations of $\SSym\nn$ is normal 
 if, and only if, the inclusion $\Des{\tau\inv}\subseteq \Des\sigma$ holds.
\end{prop}

\begin{proof}
Let $\sigma$ and $\tau$ be two signed permutations of $\SSym\nn$.
Assume that $(\sigma,\tau)$ is not normal.
Then, there exists a simple braid $z$ which is a left divisor of $r(\sigma)r(\tau)$ 
and greater than $r(\sigma)$, \ie, $r(\sigma)$ left divides $z$.
Hence, there exists $i\in[0,n]$, such that $r(\sigma)\theta_i$ is simple, 
and $\theta_i$ left divides $r(\tau)$.
Denoting by $x$ the simple braid $r(\sigma)\theta_i$ and by $y$ the positive 
braid $\theta_i\inv r(\tau)$, we obtain $r(\sigma)r(\tau)=x\,y$.

By Lemma~\ref{L:Des1}, the integer $\ii$ does not belong to $\Des\sigma$,
but in~$\Des{\tau\inv}$.
To summarize, we have proved that the pair $(\sigma,\tau)$ is not normal if there
exists $i\in[0,n]$ such that $i\not\in\Des\sigma$ and $i\in\Des{\tau\inv}$.
The conversely implication is immediate.
Therefore $(\sigma,\tau)$ is normal if, and only if, for all $i\in[0,n]$, we have either 
$i\in\Des\sigma$ or $i\not\in\Des{\tau\inv}$.
Since $\ii$ is or is not in $\Des{\tau\inv}$, we obtain that the pair~$(\sigma,\tau)$ is normal
if, and only if, $\Des{\tau\inv}\subseteq\Des\sigma$ holds, as expected.
\end{proof}

The descent set of a signed permutation $\sigma$ can be defined directly from the window notation of $\sigma$.

\begin{prop}[Proposition 8.1.2 of \cite{bjorner2006ComCoxgro}]
\label{P:Des}
 For $\nn\geq1$, $\sigma\in\SSym\nn$ and $\ii\in[0,\nno]$ we have $\ii\in\Des\sigma$ 
 if, and only if, $\sigma(\ii)>\sigma(\iip)$.  
\end{prop}

 We denote by $\QQ\SSym\nn$ the $\QQ$-vector space generated by $\SSym\nn$.
Permutations of $\SSym\nn$ are then vectors of $\QQ\SSym\nn$.
In this way, the expressions~$2\sigma$ and $\sigma+\tau$ take sense for 
$\sigma$ and $\tau$ in  $\QQ\SSym\nn$.

\begin{defi}
For $n\geq 1$, we define a square matrix $\Adj_{B_n}=(a_{\sigma,\tau})$ indexed by the elements of 
$\SSym\nn$ by 
\begin{equation*}
a_{\sigma,\tau}=\begin{cases}
                  1 & \text{for $\Des{\tau\inv}\subseteq\Des\sigma$;}\\
                  0 & \text{otherwise.}
                 \end{cases}
\end{equation*}
\end{defi}

\begin{exam} \label{exempledegre2}
There are $8$ signed permutations in $\SSym2$.
In the above table, we give them with informations about their inverse and descending sets.
\begin{equation*}
\begin{array}{c|c|c|c}
\sigma& \sigma\inv& \Des\sigma& \Des{\sigma\inv}\\
\hline 
(1,2) & (1,2) & \emptyset & \emptyset \\
(1,-2) & (1,-2) & \{1\} & \{1\} \\
(-1,2) & (-1,2) & \{0\} & \{0\} \\
(-1,-2) & (-1,-2) & \{0,1\} & \{0,1\} \\
(2,1) & (2,1) & \{1\} & \{1\} \\
(2,-1) & (-2,1) & \{1\} & \{0\} \\
(-2,1) & (2,-1) & \{0\} & \{1\} \\
(-2,-1) & (-2,-1) & \{0\} & \{0\}
\end{array}
\end{equation*}
 With the same enumeration of $\SSym2$, we obtain 
 \begin{equation*}
  \Adj_{B_2}=\begin{bmatrix}
  1 & 0 & 0 & 0 & 0 & 0 & 0 & 0\\
1 & 1 & 0 & 0 & 1 & 0 & 1 & 0\\
1 & 0 & 1 & 0 & 0 & 1 & 0 & 1\\
1 & 1 & 1 & 1 & 1 & 1 & 1 & 1\\
1 & 1 & 0 & 0 & 1 & 0 & 1 & 0\\
1 & 1 & 0 & 0 & 1 & 0 & 1 & 0\\
1 & 0 & 1 & 0 & 0 & 1 & 0 & 1\\
1 & 0 & 1 & 0 & 0 & 1 & 0 & 1
           \end{bmatrix}
 \end{equation*}

\end{exam}

\begin{lem}
\label{L:Adj}
 A pair $(\sigma,\tau)$ of signed permutation of $\SSym\nn$ is normal if, and only if, 
 the scalar ${}^t \sigma\,\Adj_{B_n}\,\tau$ is equal to $1$.
 \end{lem}

 \begin{proof}
For a pair of signed permutations $(\sigma,\tau)$, 
the scalar ${}^t \sigma\,\Adj_{B_n}\,\tau$
corresponds to the coefficient $a_{\sigma,\tau}$ of the matrix $\Adj_{B_n}$. 
We conclude by definition of $\Adj_{B_n}$ and~Proposition~\ref{P:NorDes}.
 \end{proof}

\begin{prop}
\label{P:Adj}
 Let $\sigma$ and $\tau$ be permutations of $\SSym\nn\setminus\{1\}$.
For all $d\geq 1$, the number $\nb\nn\dd(\sigma,\tau)$ of normal sequences 
 $(x_1,...,x_d)$
 with $\pi(x_1)=\sigma$ and $\pi(x_d)=\tau$ is
 \begin{equation*}
 \nb\nn\dd(\sigma,\tau)={}^t \sigma\,\Adj_{B_n}^{d-1}\,\tau.
 \end{equation*}
\end{prop}

\begin{proof}
 By induction on $d$.
 For $d=1$, such a normal sequence exists if, and only if, the permutation $\sigma$
 is equal to $\tau$.
 Hence $\nb\nn1(\sigma,\tau)$ is $\delta_\sigma^\tau$,
 which is equal to ${}^t \sigma\cdot \tau$.
 
 Assume now $d\geq 2$.
 A sequence $s=(x_1,x_2,...,x_{d-1},x_d)$ is normal 
 if, and only if, the sequence $s'=(x_1,x_2,...,x_{d-1})$ and 
 the pair $(x_{d-1},x_d)$ are normal. 
 Denoting by $\kappa$ the permutation $\pi(x_\ddo)$, we obtain
 \begin{equation*}
  \nb\nn\dd(\sigma,\tau)=\sum_{\substack{\kappa\in\SSym\nn \\ \text{$(\kappa,\tau)$ normal}}} \nb\nn\ddo(\sigma,\kappa)
 \end{equation*}
 As, by Lemma~\ref{L:Adj}, the rational ${}^t\kappa\Adj_{B_n} \tau$ is equal to $1$ if, and only if, 
 $(\kappa,\tau)$ is normal and to~$0$ otherwise, we obtain
 \begin{equation*}
  \nb\nn\dd(\sigma,\tau)=\sum_{\kappa\in\SSym\nn} \nb\nn\ddo(\sigma,\kappa) \cdot {}^t\kappa\Adj_{B_n} \tau.
 \end{equation*}
 Using induction hypothesis, we get
 \begin{align*}
\nb\nn\dd(\sigma,\tau)&=\sum_{\kappa\in\SSym\nn} {}^t\sigma(\Adj_{B_n})^{d-2}\kappa \cdot {}^t\kappa\Adj_{B_n} \tau\\
&={}^t\sigma\Adj_{B_n}^{d-2}\cdot\Adj_{B_n} \tau={}^t\sigma\Adj_{B_n}^{d-1}\tau,
 \end{align*}
 as expected.
\end{proof}

\begin{coro}
 For $n\geq 1$ and $d\geq 1$  we have 
 \begin{equation*}
  \nb\nn\dd={}^t X \Adj_{B_n}^\ddo X,
 \end{equation*}
 where $X$ is the vector $\sum_{\sigma\in\SSym\nn\setminus\{1\}}\sigma$.
\end{coro}

\begin{proof}
Let $n\geq 1$ and $d\geq 1$ be two integers.
By Proposition~\ref{P:GarPerm}, the integer $\nb\nn\dd$ is the number of 
normal sequences with no trivial entry. 
As the pair $(1,\sigma)$ is never normal for $\sigma\in\SSym\nn$, 
a sequence $(x_1,...,x_d)$ is not normal whenever $x_i=1$ for any $i$ in $[2,\ddo]$.
Hence, $\nb\nn\dd$ is the number of normal sequences $(x_1,...,x_d)$ 
with $x_1\not=1$ and $x_d\not=1$:
\begin{equation*}
 \nb\nn\dd=\sum_{\sigma,\tau\in\SSym\nn\setminus\{1\}} \nb\nn\dd(\sigma,\tau). 
\end{equation*}
which is equal, by Proposition~\ref{P:Adj}, to 
\begin{equation*}
 \nb\nn\dd=\sum_{\sigma,\tau\in\SSym\nn\setminus\{1\}} {}^t\sigma\,\Adj_{B_n}^\ddo\,\tau={}^t X\,\Adj_{B_n}^\ddo X,
\end{equation*}
as expected.
\end{proof}

\begin{exam}
 In $\BPB2$ the only braid of Garside length $0$ is the trivial one, \ie, $\nb20=1$.
 Except the trivial one, all simple braids have length $1$, and so $\nb21=7$, corresponding to ${}^t XX$.
 Considering the matrix~$\Adj_{B_n}$ we obtain the following values of $\nb\nn\dd$: 
 
 \[\begin{array}{c|c|c|c}
    d& \nb2d & \nb3d & \nb4d\\
    \hline
    0 	&	1 	&	47		&	383\\
    1 	&	7	&	771		&	35841\\
    2 	&	25	&	10413		&	2686591\\
    3 	& 	79	&	134581		&	193501825\\
    4 	&	241	&	1721467		&	13837222655\\
    5 	& 	727	&	21966231 	&	988224026497
    \end{array}
\]	
The generating series $F_{B_n}(t)=\displaystyle\sum_{d=0}^{+\infty} \nb\nn{d} t^d$ is given by ${}^tX \left(\mathrm{I}-t\Adj_{B_2}\right)\inv X$:
 \begin{align*}
  F_{B_2}(t)&=\frac{7-3t}{(3t-1)(t-1)}\\
  F_{B_3}(t)&=\frac{-60t^4+149t^3-163t^2+169t-47}{(t-1)(3t-1)(20t^3-43t^2+16t-1)}
\end{align*}
Developing  $F_{B_2}(t)$, we obtain $\nb2d=3^\ddp-2$.
\end{exam}

The eigenvalues of the matrix $\Adj_{B_n}$ give informations on the growth of the function $d\mapsto \nb\nn\dd$. 
The first point is to determine if the eigenvalues of $\Adj_{B_\nno}$ are also eigenvalues
of $\Adj_{B_n}$, \ie, to determine if the characteristic polynomial of the matrix $\Adj_{B_\nno}$
divides the one of $\Adj_{B_n}$.
In \cite{dehornoy2007Comnorseqbra}, P. Dehornoy conjectured that this divisibility result holds for classical braids (Coxeter type A). 
The conjecture was proved by F. Hivert, J.-C. Novelli and J.Y. Thibon in \cite{hivert2008SurunecondeDeh}.
If we denote by $\chi_\nn$ the characteristic polynomial of the matrix $\Adj_{B_n}$, we obtain:
\begin{align*}
\chi_1(x)&=(x-1)^2\\
\chi_2(x)&=\chi_1(x)\ x^4\  (x-1)\ (x-3) \\
\chi_3(x)&=\chi_2(x)\ x^{37}\ (x^3-16x^2+43x-20)\\
\chi_4(x)&=\chi_3(x)\ x^{329}\ (x-1)^3\ (x^4-85x^3+1003x^2-2291x+1260)\\
\chi_5(x)&=\chi_4(x)\ x^{3449}\ (x^7-574x^6+39344x^5-576174x^4+\\
& \phantom{\chi_4(x)\ x^{3449}\ (x^7-574} 3027663x^3-5949972x^2+4281984x-1088640)
\end{align*}
As the reader can see, the polynomial $\chi_\ii$ divides $\chi_\iip$ for $\ii\in\{1,2,3,4\}$.
The aim of the paper is to prove the following theorem:

\begin{thrm}
\label{T:Main}
For all $\nn\in\NN$, the characteristic polynomial of the matrix $\Adj_{B_n}$ 
divides the characteristic polynomial of the matrix $\Adj_{B_\nnp}$.
\end{thrm}

For this, we interpret the matrix $\Adj_{B_n}$ as the matrix of an endomorphism $\Phi_\nn$ of~$\QQ\SSym\nn$.
In order to prove the main theorem we equip the vector space $\QQ\SSym\nn$ with a structure of Hopf algebra.

\section{The Hopf algebra $\BFQSym$.}

We describe in this section  an analogous of the Hopf algebra $\FQSym$ for the signed permutation group $\SSym\nn$. 
We denote by $\QQ\SSym{}$ the $\QQ$-vector space $\bigoplus_{n=1}^{+\infty}\QQ\SSym\nn$.

\subsection{Signed permutation words}

We have shown in Section~\ref{S:Signed} that a signed permutation can be uniquely determined by its window notation.
In order to have a simple definition for the notions attached to the construction of the Hopf algebra $\BFQSym$, 
we describe a one-to-one construction between signed permutations and some specific words associated to the window notation.

\begin{defi}
 For $\nn\geq 1$, we define $\WW\nn$ to be the set of words $w=w_1\,...\,w_n$ 
 on the alphabet $[-\nn,\nn]$  satisfying $\{|w_1|,...,|w_\nn|\}=[1,\nn]$.
\end{defi}

If $w$ is an element of $\WW\nn$, then $(w_1,...,w_\nn)$ is the window notation of some signed permutation of $\SSym\nn$.
For $\nn\geq 1$, we define two maps $w:\SSym\nn\to\WW\nn$ and $\perm:\WW\nn\to\SSym\nn$ by
$w(\sigma)=\sigma(1)\,...\,\sigma(\nn)$ and, for $i\in[-n,n]$,
\begin{equation*}
\perm(w)(i)=\begin{cases}
             0 & \text{if $i=0$,}\\
             w_i & \text{if $i>0$,}\\
             -w_{-i} & \text{if $i<0$.}
            \end{cases}
\end{equation*}

\begin{defi}
 For $\ii\in\ZZ\setminus\{0\}$ and $\kk\in\ZZ$, 
 we define the integers $\shift\ii\kk$ and $\dec\ii\kk$ (whenever $\ii\not=\pm\kk$) by
 \begin{equation*}
  \shift\ii\kk=\begin{cases}
            \ii+\kk&\text{if $\ii>0$,}\\
            \ii-\kk&\text{if $\ii<0$.}
           \end{cases}
  \qquad
  \dec\ii\kk=\begin{cases}
	      \iip&\text{if $\ii<-\kk$,}\\
              \ii&\text{if $-\kk<\ii<\kk$,}\\
              \iio&\text{if $\ii>\kk$.}
             \end{cases}
 \end{equation*}
 For $\ww=\ww_1\,...\,\ww_\ell$ a word on the letters $[-\nn,\nn]\setminus\{0\}$, we define 
 $\shift\ww\kk$ to be the word $\shift{\ww_1}\kk\,...\,\shift{\ww_\ell}\kk$ and
 $\dec\ww\kk$ to be the word $\dec{\ww_1}\kk\,...\,\dec{\ww_\ell}\kk$ if $\ww_j \neq \pm k$ for all $j$. 
 We also extend these notations to sets of integers.
\end{defi}

\begin{exam}
If $w$  is the word $1\cdot-5\cdot3\cdot-2\cdot6$, we have $\shift\ww2=3\cdot-7\cdot5\cdot-4\cdot 8$ and $\dec\ww4=1\cdot-4\cdot3\cdot-2\cdot 5$.
 \end{exam}

\subsection{Shuffle product}

\begin{defi}
For $\kk,\ell\geq 1$, we denote by $\SH\kk\ell$ all the subsets of $[1,\kk+\ell]$ of cardinality $\kk$.
For $X\in\SH\kk\ell$, we write $X=\{x_1<...<x_\kk\}$ to specify that the $x_i$'s are the elements of $X$ in increasing order. 
\end{defi}

For example, we have 
\begin{equation*}
\SH23=\{\{1,2\},\{1,3\},\{1,4\},\{1,5\},\{2,3\},\{2,4\},\{2,5\},\{3,4\},\{3,5\},\{4,5\}\}.
\end{equation*}

\begin{defi}
Let $\kk,\ell\geq 1$ be two integers.
For two words  $\uu\in\WW\kk$,  $\vv\in\WW\ell$ and $X=\{x_1<...<x_\kk\}\in\SH\kk\ell$ we define 
the \emph{$X$-shuffle word of $\uu$ and $\vv$} by 
\begin{equation*}
\uu\shuffle^X\vv=\vv^0[\kk]\,\uu_1\,\vv^1[\kk]\,...\,\vv^\kko[\kk]\,\uu_\kk\,\vv^\kk[\kk]
\end{equation*}
where $\vv^0\,...\,\vv^\kk=\vv$ and $\len(\vv^i)=x_\iip-x_\ii-1$, with the conventions
$x_0=0$ and $x_\kkp=\kk+\ell$.
 \end{defi}

One remarks that letters coming from $\uu$ are in positions belonging to $X$ in the final word.
 \def\uuf{{\color{gray}-2}}
 \def\uus{{\color{gray}1}}

 \begin{exam}
Let $\uu$ be the word $-2\cdot1$ and $\vv$ be the word $3\cdot-1\cdot 2$. We then have $\kk=2$ and $\ell=3$.
The word $\shift\vv\kk$ is $5\cdot-3\cdot 4$. The $\{2,4\}$-shuffle of $\uu$ and $\vv$ is the word 
$5\cdot\uuf\cdot-3\cdot\uus\cdot4$ while the $\{4,5\}$-shuffle of $\uu$ and
$\vv$ is $5\cdot-3\cdot4\cdot\uuf\cdot\uus$; letters in gray are these coming from the word $\uu$.
 \end{exam}

\begin{defi}
 For $\sigma\in\SSym\kk$ and $\tau\in\SSym\ell$ two signed permutations, 
 we define the \emph{shuffle product} of $\sigma$ and $\tau$ is the signed 
 permutation $\sigma\shuffle\tau$ of $\SSym{\kk+\ell}$ defined by
 \[
  \sigma\shuffle\tau=\sum_{X\in\SH\kk\ell} \perm\left(w(\sigma)\shuffle^Xw(\tau)\right)
 \]
\end{defi}

\begin{exam}
Considering the signed permutations $\sigma=(-2,1)$ and $\tau=(3,-1,2)$, we obtain
 \begin{align*}
\sigma\shuffle\tau=\ (\uuf,\uus,5,-3,4)+&(\uuf,5,\uus,-3,4)+(\uuf,5,-3,\uus,4)+(\uuf,5,-3,4,\uus)\\
+(5,\uuf,\uus,-3,4)+&(5,\uuf,-3,\uus,4)+(5,\uuf,-3,4,\uus)+(5,-3,\uuf,\uus,4)\\
+(5,-3,\uuf,4,\uus)+&(5,-3,4,\uuf,\uus).
 \end{align*}
\end{exam}

Let $x_1,...,x_\nn$ be $n$ distinct integers. 
For every sequence $\varepsilon_1,...,\varepsilon_\nn$ of $\{-1,+1\}$, 
we define $\std(\varepsilon_1\, x_1\,...\,\varepsilon_\nn\, x_\nn)$
to be the word $\varepsilon_1\, f(x_1)\,...\,\varepsilon_\nn\, f(x_\nn)$, 
where $f$ is the unique increasing map
from $\{x_1,...,x_\nn\}$ to $[1,\nn]$.
Apart from the $\varepsilon_\ii$, this notion of standardization of 
word coincides with the one used on permutations of $\SSym\nn$.

We define a coproduct on $\QQ\SSym{}$ by
\begin{equation*}
\forall \sigma\in\SSym\nn, \quad \Delta(\sigma)=\sum_{k=0}^\nn \rho(\std(\sigma(1),...,\sigma(\kk)))\otimes\rho(\std(\sigma(\kkp),...,\sigma(\nn))) 
\end{equation*}
For example the coproduct of $(4,-2,3,-1)$ is 
\begin{align*}
 \Delta(4,-2,3,1)=&\emptyset\otimes(4,-2,3,1)+(1)\otimes(-2,3,1)\\
 &+(2,-1)\otimes(2,1)+(3,-1,2)\otimes(1)+(4,-2,3,1)\otimes\emptyset
\end{align*}
Equipped with the shuffle product $\shuffle$ and the coproduct $\Delta$, the 
vector space~$\QQ\Sym{}$ is a Hopf algebra denoted $\BFQSym$.
Details are omitted in this paper and can be found in \cite{novelli2010Frequafundesalgwreprononmulfun}.
Indeed, $\BFQSym$ corresponds to the Hopf algebra of decorated permutations $\FQSym^D$ with $D=\{-1,1\}$. 

\subsection{The dual structure.}

 Thanks to the non degenerating coupling~$\left<\sigma,\tau\right>=\delta_\sigma^\tau$,
 we identify $\BFQSym$ with its dual.
 The Hopf algebra structure of the dual is given by the product $\ast$ 
 and the coproduct $\delta$ defined by:
 \begin{equation*}
  \left<\sigma\ast\tau,\kappa\right>=\left<\sigma\otimes\tau,\Delta(\kappa)\right>
  \qquad\text{and}\qquad
  \left<\delta(\sigma),\tau\otimes\kappa\right>=\left<\sigma,\tau\shuffle\kappa\right>.
 \end{equation*}

The map $\iota$ of $\QQ\SSym{}$ that maps $\sigma$ to $\sigma\inv$
is a Hopf algebra isomorphism between $(\BFQSym,\shuffle,\Delta)$ and $(\BFQSym,\ast,\delta)$.
The following proposition gives a concrete description of $\ast$.

\begin{prop}
\label{P:CoProduct}
Let $\sigma\in\SSym\kk$ and $\tau\in\SSym\ell$ be two permutations. We have
 \begin{equation*}
  \sigma\ast\tau\ =\hspace{-2em}\sum_{
    \substack{\uu\in\WW{\kk+\ell} \\ 
      \std(u_1,...,u_k)=w(\sigma) \\ 
      \std(u_\kkp,...,u_{\kk+\ell})=w(\tau) }} \perm(\uu)
 \end{equation*}
\end{prop}

\begin{exam}
 For the signed permutations $\sigma=(2,-1)$ and $\tau=(3,-1,2)$ we have
 \begin{align*}
  \sigma\ast\tau=\ ({\color{gray}2},{\color{gray}-1},5,-3,4)&+({\color{gray}3},{\color{gray}-1},5,-2,4)+({\color{gray}4},{\color{gray}-1},5,-2,3)+({\color{gray}5},{\color{gray}-1},4,-2,3)\\
  +({\color{gray}3},{\color{gray}-2},5,-1,4)&+({\color{gray}4},{\color{gray}-2},5,-1,3)+({\color{gray}5},{\color{gray}-2},4,-1,3)+({\color{gray}4},{\color{gray}-3},5,-1,2)\\
  +({\color{gray}5},{\color{gray}-3},4,-1,2)&+({\color{gray}5},{\color{gray}-4},3,-1,2).
 \end{align*}

\end{exam}

\begin{defi}
For $\nn\geq 1$, we denote by  $I_n$, $J_n$, $P_n$ and $Q_n$ the elements of $\QQ\SSym\nn$ defined by 
 $I_\nn=(1,...,\nn)$, $J_\nn=(-\nn,...,-1)$, and 
 \begin{equation*}
 P_\nn=\hspace{-0.5em}\sum_{\substack{\sigma\in\SSym\nn \\ \Des{\sigma\inv}\subseteq\{0\}}}\sigma,
  \qquad 
  Q_\nn=\hspace{-0.5em}\sum_{\substack{\sigma\in\SSym\nn \\ \Des{\sigma}\subseteq\{0\}}}\sigma.
  \end{equation*}
 \end{defi}

 \begin{exam}
 \label{E:Pn}
  We have $P_2=(1,2)+(-1,2)+(2,-1)+(-2,-1)$, $Q_2=(1,2)+(-1,2)+(-2,1)+(-2,-1)$ and, for example:
  \begin{align*}
  P_4=&(1,2,3,4)+(-1,2,3,4)+(2,-1,3,4)+(-2,-1,3,4)+(2,3,-1,4)\\
       &+(-2,3,-1,4)+(2,3,4,-1)+(-2,3,4,-1)+(3,-2,-1,4)\\
       &+(-3,-2,-1,4)+(3,-2,4,-1)+(-3,-2,4,-1)+(3,4,-2,-1)\\
       &+(-3,4,-2,-1)+(4,-3,-2,-1)+(-4,-3,-2,-1)
  \end{align*}
In general, $P_\nn$ and $Q_\nn$ are linear combinations of $2^n$ permutations.
  
 \end{exam}

Vectors $P_\nn$ and $Q_\nn$ are used to describe permutations of $\SSym\nn$ whose descent sets are included in a given subset.
The following Lemma exhibits these connections.

 \begin{lem}
 \label{L:Prods}
 Let  $\kk_1,...,\kk_\ell\geq 1$ be integers.
 Denoting by $n$ the integer $\kk_1+...+\kk_\ell$ and by $D$ the set $\{\kk_1,\kk_1+\kk_2,...,\kk_1+...+\kk_{\ell-1}\}$,
 we have the following relations:
 \begin{align*}
 Q_{\kk_1}\ast...\ast Q_{\kk_\ell}&\ =\hspace{-0.7em}\sum_{\substack{
  \sigma\in\SSym\nn \\ 
  \Des{\sigma}\subseteq \{0\}\cup D
  }} \hspace{-1.1em}\sigma, 
 &I_{\kk_1}\ast Q_{\kk_2}\ast...\ast Q_{\kk_\ell}&\ =\hspace{-0.7em}\sum_{\substack{
  \sigma\in\SSym\nn \\ 
  \Des{\sigma}\subseteq D
  }} \hspace{-0.6em}\sigma,\\
  P_{\kk_1}\shuffle...\shuffle P_{\kk_\ell}&\ =\hspace{-1.3em}\sum_{\substack{
  \sigma\in\SSym\nn \\
  \Des{\sigma\inv}\subseteq \{0\}\cup D
  }} \hspace{-1.5em}\sigma,
  &I_{\kk_1}\shuffle P_{\kk_2}\shuffle...\shuffle P_{\kk_\ell}&\ =
  \hspace{-1.2em}\sum_{\substack{
  \sigma\in\SSym\nn \\ 
  \Des{\sigma\inv}\subseteq  D
  }} \hspace{-1em}\sigma.
\end{align*}
 \end{lem}
 
 \begin{proof}
 For $\ii\in[1,\ell]$ we put $d_\ii=\kk_1+...+\kk_\ii$. By very definition of $Q_\kk$, we have
  \begin{equation*}
   Q_\kk=\hspace{-1.2em}\sum_{\substack{\sigma\in\SSym\kk \\ \sigma(1)<...<\sigma(\kk)}} \hspace{-0.5em}\sigma.
  \end{equation*}
Then, by Proposition~\ref{P:CoProduct}, we obtain
\begin{equation*}
 Q_{\kk_1}\ast...\ast Q_{\kk_\ell}\ =\hspace{-2.2em}\sum_{\substack{
  \sigma\in\SSym{\kk+\ell},\\
  \sigma(1)<...<\sigma(\dd_1),
  \\ ... \\
  \sigma(\dd_{\ell-1}+1)<...<\sigma(\dd_\ell)}}
  \hspace{-2em}\sigma.
\end{equation*}
Permutations occurring in the previous sum are exactly these having descents in the set $\{0,\dd_1,...,\dd_{\ell-1}\}$.
Similarly, as $I_{\kk_1}$ is the only permutation $\sigma$ of $\SSym{\kk_1}$ satisfying $0<\sigma(1)<...<\sigma(\kk_1)$, we have
\begin{equation*}
 I_{\kk_1}\ast Q_{\kk_2}\ast...\ast Q_{\kk_\ell}\ =
 \hspace{-2em}
 \sum_{\substack{
  \sigma\in\SSym{\kk+\ell},\\
  0<\sigma(1)<...<\sigma(\dd_1),\\
  \sigma(\dd_1+1)<...<\sigma(\dd_2),
  \\ ... \\
  \sigma(\dd_{\ell-1}+1)<...<\sigma(\dd_\ell)}}
  \hspace{-1.5em}\sigma,
\end{equation*}
which is the sum of permutations of $\SSym\nn$ with descent set in $\{d_1,...,d_{\ell-1}\}$. 

Applying the isomorphism $\iota$ between $(\BFQSym,\shuffle,\Delta)$ and $(\BFQSym,\ast,\delta)$ 
to the previous expression of $Q_{\kk_1}\ast...\ast Q_{\kk_\ell}$, we obtain

\begin{equation*}
\iota(Q_{\kk_1})\shuffle...\shuffle\iota(Q_{\kk_\ell})\ =\hspace{-1.2em}
\sum_{\substack{\sigma\in\SSym\nn \\ \Des\sigma\in\{0\}\cup D}} \hspace{-1em}\sigma\inv\hspace{0.5em}=\hspace{-1em}\sum_{\substack{\sigma\in\SSym\nn \\ \Des{\sigma\inv}\in\{0\}\cup D}} \hspace{-1em}\sigma
\end{equation*}
The expected relation appears, remarking that $\iota(Q_\kk)$ is equal to $P_\kk$. 
The second relation involving the shuffle product is obtain similarly from $\iota(I_{\kk_1})=I_{\kk_1}$.
 \end{proof}
 
The vector $P_\nn$ of $\QQ\SSym\nn$ can also be defined using the shuffle product as suggested by~Example~\ref{E:Pn}.
\begin{lem}
\label{L:Pshuffle}
 For all $\nn\geq 1$, we have 
$P_n=\sum_{k=0}^n J_\kk\shuffle I_{\nn-\kk}$.
 \end{lem}

\begin{proof}
 By definition of $Q_\nn$, we have 
 \begin{equation*}
  Q_\nn\ =\hspace{-1.3em}\sum_{\substack{\sigma\in\SSym\nn, \\ \sigma(1)<...<\sigma(\nn)}}\hspace{-0.8em}\sigma \hspace{0.5em}=\hspace{0.5em} 
  \sum_{k=0}^\nn \hspace{-0.5em}\sum_{\substack{\sigma\in\SSym\nn, \\\sigma(1)<...<\sigma(\kk)<0\\
  0<\sigma(\kkp)<...<\sigma(\nn)}}\hspace{-1em}\sigma
 \end{equation*}
By Proposition~\ref{P:CoProduct}, we have
\begin{equation*}
J_\kk\ast I_{\nn-\kk} \hspace{1em}= \hspace{-2em}
 \sum_{\substack{
 \sigma\in\SSym\nn \\ 
 \std(\sigma(1),...,\sigma(\kk))=(-\kk,...,-1) \\ 
 \std(\sigma(\kk+1),...,\sigma(\nn))=(1,...,\nn-\kk)}}
 \hspace{-1.5em}\sigma\hspace{1em}=\hspace{-0.5em}
 \sum_{\substack{\sigma\in\SSym\nn, \\\sigma(1)<...<\sigma(\kk)<0\\
  0<\sigma(\kkp)<...<\sigma(\nn)}}\hspace{-1.5em} \sigma.
\end{equation*}
We have then established $Q_\nn=\sum_{\kk=0}^\nn J_\kk\ast I_{\nn-\kk}$.
We obtain the expected result applying the isomorphism $\iota$ since  $J_\kk$ and $I_\kk$ are fixed by $\iota$.
\end{proof}

\section{The divisibility result.}

For $n\in\NN$, we define  $\Phi_n$ to be the endomorphism of $\QQ\SSym\nn$
whose representative matrix is ${}^t\Adj_{B_\nn}$.
We denote by $\Phi$ the endomorphism $\bigoplus\Phi_\nn$ of~$\QQ\SSym{}$.
By very definition of $\Adj_{B_\nn}$, for all $\sigma\in\SSym\nn$, we have
\begin{equation*}
 \Phi(\sigma)=\Phi_\nn(\sigma)\ =\hspace{-1.5em}\sum_{\substack{\tau\in\SSym\nn\\ \Des{\tau\inv}\subseteq\Des\sigma}} \tau.
\end{equation*}

For $\nn\in\NN$, we denote by $\mathcal{D}_\nn$ the set of all subsets of $[0,\nno]$. 
The descent map from $\SSym\nn$ to $\mathcal{D}_\nn$ can be extended to a 
unique linear map, also denoted by $\mathrm{Des}$, from~$\QQ\SSym\nn$ to $\QQ\mathcal{D}_\nn$.
We denote by $\widetilde{\Phi}_\nn$ the map from $\QQ\mathcal{D}_\nn$ to $\QQ\SSym\nn$ defined by 
\[
 \widetilde{\Phi}_\nn(I)\ =\hspace{-1em}\sum_{\substack{\tau\in\SSym\nn\\\Des{\tau\inv}\subseteq I}} \tau,
\]
for any element $I$ of $\mathcal{D}_\nn$.
For all $\sigma\in\SSym\nn$, we have $\Phi_\nn(\sigma)=\widetilde{\Phi}_\nn(\Des\sigma)$.

A direct consequence of Lemma~\ref{L:Prods} is :

 \begin{prop}
 \label{P:PhiTilde}
For every $D=\{d_1<...<d_\ell\}$ element of $\mathcal{D}_\nn$, with $0<d_1$, we have the relations
  \begin{align*}
   \widetilde{\Phi}_\nn(D)&=I_{k_1}\shuffle P_{k_2} \shuffle ... \shuffle P_{\kk_\ell}\\
   \widetilde{\Phi}_\nn(\{0\}\cup D)&= P_{k_1}\shuffle P_{k_2} \shuffle... \shuffle P_{\kk_\ell}
  \end{align*}
where $k_i=d_\iip-d_\ii$ for $\ii\in[1,\nn]$ and  with the convention $\dd_{\ell+1}=\nn$.
 \end{prop}
 
\begin{defi}
An  endomorphism $\Psi$ of $\QQ\SSym{}$  is a \emph{surjective derivation} if

-- $(i)$ $\Psi(x\shuffle y)=\Psi(x)\shuffle y+x\shuffle\Psi(y)$ holds 
for all $x$, $y$ of $\QQ\SSym{}$;

-- $(ii)$ $\Psi(\QQ\SSym\nn)=\QQ\SSym\nno$ holds for all $n\geq 1$. 
\end{defi}

\begin{prop}
\label{P:Div}
 If there exists a surjective derivation $\Psi$ of $\QQ\SSym{}$ commuting with~$\Phi$, 
 then, for $n\geq 1$, the characteristic polynomial of $\Phi_\nno$ divides the one of~$\Phi_\nn$.
\end{prop}

\begin{proof}
Let $\Psi$ be a derivation of $\QQ\SSym{}$ commuting with $\Psi$, and $\nn$ be an 
integer greater than $1$.
Let us denote by $\Psi_\kk$ the restriction of $\Psi$ to $\QQ\SSym\kk$ for $\kk\in\NN$.
We fix a basis~$\mathcal{B}=\mathcal{B}_0\sqcup\mathcal{B}_1$ of $\QQ\SSym\nn$, 
such that $\mathcal{B}_0$ is a basis of $\ker(\Psi_\nn)$.
Restricting the relation $\Psi\circ\Phi=\Phi\circ\Psi$ to $\QQ\SSym\nn$, we obtain 
$\Psi_\nn\circ\Phi_n=\Phi_\nno\circ\Psi_\nn$. 
For $x$ in $\ker(\Psi_\nn)$, we have $\Psi_\nn(\Phi_\nn(x))=\Phi_\nno(\Psi_\nn(x))=\Phi_\nno(0)=0$.
Hence, $\ker(\Psi_n)$ is stable under the map $\Phi_\nn$. 
In particular, the representative matrix of $\Phi_\nn$ in the basis $\mathcal{B}$ is the upper triangular matrix
\begin{equation*}
 M_\nn=\begin{bmatrix}
        A_n & B_n \\
        0 & C_n
       \end{bmatrix}
\end{equation*}
Denoting by $\chi(.)$ the characteristic polynomial of a matrix or an endomorphism,
we obtain 
\begin{equation}
\label{E:RelCharPoly}
 \chi(\Phi_n)=\chi(M_n)=\chi(A_n)\chi(C_n).
 \end{equation}
The matrix of the restriction $\overline{\Phi}_\nn$ of $\Phi_\nn$ to 
$\QQ\SSym\nn/\ker(\Psi_n)$ is $C_n$ and so $\chi(\overline{\Phi}_n)$ is equal to $\chi(C_n)$.
 From the surjectivity of $\Psi$, we have the following commutative diagram:
$$\xymatrix{\QQ\SSym\nn/\ker(\Psi_n)\ar[r]^{\overline{\Phi}_n} \ar@{_(->>}[d]_{\overline{\Psi}_n}&\QQ\SSym\nn/\ker(\Psi_n)\ar@{^(->>}[d]^{\overline{\Psi}_n}\\
\QQ\SSym\nno\ar[r]^{\Phi_\nno}&\QQ\SSym\nno}$$
%\begin{equation*}
%\begin{tikzcd}
%\QQ\SSym\nn/\ker(\Psi_n) \arrow[r, "\overline{\Phi}_n"] \arrow[d,hook',two heads,"\overline{\Psi}_n"{left}]& \QQ\SSym\nn/\ker(\Psi_n) \arrow[d,hook,two %heads,"\overline{\Psi}_n"]\\
%\QQ\SSym\nno \arrow[r, "\Phi_\nno"]& \QQ\SSym\nno
%\end{tikzcd}
%\end{equation*}
implying that the endomorphism $\Phi_\nno$ is conjugate to $\overline{\Phi}_n$.
Therefore Equation~\eqref{E:RelCharPoly} becomes $\chi(\Phi_\nn)=\chi(A_n)\chi(\Phi_\nno)$, and so $\chi(\Phi_\nno)$ divides $\chi(\Phi_\nn)$.
\end{proof}

As the reader can check, the property $(i)$ of a derivation is not used in the proof, but
will be fundamental in order to establish the commutativity with $\Phi$.

It remains to construct a surjective derivation $\Psi$ which commutes with $\Phi$.

\subsection{A derivation of $\BFQSym$.}

In order to describe our derivation, we have to introduce some notations. 

\begin{defi}
For $a$ and $b$ two distinct integers, we define $\eps(a,b)$ by
\begin{equation*}
 \eps(a,b)=\begin{cases}
                1 & \text{if $a<b$,}\\
                -1 & \text{if $a>b$.}
               \end{cases}
\end{equation*}
For $a,b,c$ three distinct integers, we write $\eps(a,b,c)=\dfrac12\left(\eps(a,b)+\eps(b,c)\right)\in\{-1,0,1\}$.
\end{defi}

\begin{defi}
Let $\uu=\uu_1\,...\,\uu_\nn$ be a word of $\WW\nn$ and $\ii\in[1,\nn]$. We define
\[
\sign\ii(\uu)=\eps(\uu_\jjo,\uu_\jj,\uu_\jjp), 
 \]
where $\jj$ is the unique integer satisfying $|\uu_\jj|=\ii$, with the conventions $\uu_0=0$ and $\uu_\nnp=-\infty$.
\end{defi}

\begin{exam}
\label{E:Sign1}
 Considering the word $u=-1\cdot2\cdot-4\cdot-5\cdot3\cdot6$, augmented 
 to the word $0\cdot-1\cdot2\cdot4\cdot-5\cdot3\cdot6\cdot-\infty$, we obtain
 \begin{align*}
  \sign1(\uu)&=\eps(0,-1,2)=0, & \sign2(\uu)&=\eps(-1,2,-4)=0,\\
  \sign3(\uu)&=\eps(-5,3,6)=1, & \sign4(\uu)&=\eps(2,-4,-5)=-1,\\
  \sign5(\uu)&=\eps(-4,-5,3)=0, & \sign6(\uu)&=\eps(3,6,-\infty)=0.
 \end{align*}

\end{exam}

\begin{lem}
\label{L:Sign:E}
Let $\nn\geq 1$ and $\sigma\in\SSym\nn$.
 For $\jj\in[1,\nno]$, we have 
 \begin{equation*}
  \sign{|\sigma(\jj)|}(\ww(\sigma))=\begin{cases}
                       1 & \text{for $\{\jjo,\jj\}\cap\Des\sigma=\emptyset$};\\
                       -1 & \text{for $\{\jjo,\jj\}\subseteq\Des\sigma$};\\
                       0 & \text{otherwise.}
                      \end{cases}
 \end{equation*}
 Moreover, the value of $\sign{|\sigma(\nn)|}(\ww(\sigma))$ is $-1$ if $\nno$ belongs to $\Des\sigma$,
 and is $0$ otherwise.
\end{lem}

\begin{proof}
Let $\sigma$ be a permutation of $\SSym\nn$ and $\jj$ be an integer of $[1,\nno]$.
We denote by $\ii$ the integer $|\sigma(\jj)|$.
Assume first $\jj\in[1,\nno]$.
By definition of $\sign{}$, we have  $\sign\ii(c)=1$ if, and only if,
 $\sigma(\jjo)<\sigma(\jj)<\sigma(\jjp)$, wich is equivalent to 
 $\jjo\not\in\Des\sigma$ and $\jj\not\in\Des\sigma$.
 Always by definition of $\sign{}$, we have $\sign\ii(\sigma)=-1$ if, and only if,
 $\sigma(\jjo)>\sigma(\jj)>\sigma(\jjp)$, \ie, $\jjo$ and $\jj$ belong to $\Des\sigma$.

Let us now prove the result for $\ii=\nn$. 
As the relation $\sigma(\nn)>-\infty$ is always true, the value of $\sign\ii(\ww(\uu))$ is $-1$ if 
$\sigma(\nno)>\sigma(\nn)$ and $0$ otherwise, corresponding to the statement.
\end{proof}

\begin{exam}
 The descent set of $\sigma=(-1,2,-4,-5,3,6)$ is $\{0,2,3\}$.
 Hence, the non zero values of $\sign{|\sigma(\jj)|}$ are obtained 
 for $\jj=3$ and $\jj=5$, more precisely, we have $\sign{|\sigma(3)|}(\sigma)=\sign{4}(\sigma)=-1$ 
 and $\sign{|\sigma(5)|}(\sigma)=\sign{3}(\sigma)=1$, corresponding to Example~\ref{E:Sign1}.
\end{exam}

\begin{defi}
For $\uu\in\WW\nn$ and $\ii\in[1,\nn]$, we denote by $\del\ii(\uu)$ the word 
$\dec{\uu_1}\ii\,...\,\dec{\uu_\jjo}\ii\,\dec{\uu_\jjp}\ii\,...\,\dec{\uu_\nn}\ii$ of $\WW\nno$,
where $\jj$ is the unique integer satisfying the relation $|\uu_\jj|=\ii$.
\end{defi}

\begin{exam}
 Considering $\uu=-1\cdot2\cdot-4\cdot-5\cdot3\cdot6$, we obtain 
 \begin{align*}
  \del1(\uu)&=1\cdot-3\cdot -4 \cdot 2\cdot 5&\del2(\uu)&=-1\cdot-3\cdot-4\cdot2\cdot5\\
  \del3(\uu)&=-1\cdot2\cdot-3 \cdot -4 \cdot5&\del4(\uu)&=-1\cdot2\cdot-4\cdot3\cdot5\\
  \del5(\uu)&=-1\cdot2\cdot-4\cdot3\cdot5&\del6(\uu)&=-1\cdot2\cdot-4\cdot-5\cdot3.
  \end{align*}

\end{exam}

\begin{defi}
Let $\nn$ and $\ii$ be two integers such that $\ii\in[1,\nn]$.
We define a linear map  $\partial_n^i$ of $\QQ\SSym\nn$ to $\QQ\SSym\nno$
by 
\begin{equation*}
\partial_n^i(\sigma)=\sign\ii(\ww(\sigma))\,\rho(\del\ii(\ww(\sigma))),
 \end{equation*}
 where $\sigma\in\SSym\nn$.
Then we define a map $\partial_\nn$ from $\QQ\SSym\nn$ to $\QQ\SSym\nno$ by
 \begin{equation*}
   \partial_\nn(\sigma)=\sum_{k=1}^\nn\partial^\ii_\nn(\sigma)\qquad\text{for $\sigma\in\SSym\nn$},
 \end{equation*}
and a map $\partial$ of $\QQ\SSym{}$ by $\partial=\bigoplus_{\nn=1}^{+\infty}\partial_\nn$. 
\end{defi}

% \begin{exam}
%  For $\sigma=(2,1,-3,4,5)$ and $i=2$, we have $v_2(\sigma)=\{2\}$ 
%  and $\sign2(\sigma)=-1$.
%  As the permutation $\sigma^{[1,5]\setminus\{2\}}$ is the standardized of
%  $(1,-3,4,5)$, which is $(1,-2,3,4)$ we obtain 
%  $\partial_5^2((2,1,-3,4,5))=-(1,-2,3,4)$.
% \end{exam}(-1,2,-4,-5,3,6)

\begin{exam}
 Considering the permutation $\sigma=(-1,2,-4,-5,3,6)$, we have $\partial_6^1(\sigma)=\partial_6^2(\sigma)=\partial_6^5(\sigma)=\partial_6^6(\sigma)=0$,
while  we have
\begin{align*}
\partial_6^3(\sigma)&=\sign3(\sigma)\rho(\del3(\ww(\sigma)))=(-1,2,-3,-4,5)\\
\partial_6^4(\sigma)&=\sign4(\sigma)\rho(\del4(\ww(\sigma)))=-(-1,2,-4,3,5)
\end{align*}
Finally we obtain $\partial(\sigma)=(-1,2,-3,-4,5)-(-1,2,-4,3,5)$.
 \end{exam}

\begin{exam} The map $\partial$ sends $\QQ\SSym{2}$ to $\QQ\SSym{1}$. The matrix of this map, with the enumeration of $\SSym{2}$
of Example \ref{exempledegre2} and the enumeration $(1)$, $(-1)$ of $\SSym{1}$, is:
$$\begin{bmatrix}
1&-1&0&0&-1&-1&0&0\\
0&0&0&-2&0&0&0&0
\end{bmatrix}$$
\end{exam}

 We now prove that $\partial$ is a surjective derivation of $\QQ\SSym{}$, compatible with the shuffle product.

\begin{lem}
\label{L:CompPart}
 Let $\sigma\in\SSym\kk$ and $\tau\in\SSym\ell$ be two signed permutations. 
 \vspace{0.2em}
 
 -- $(i)$ For all $i\in[1,\kk]$, we have $\partial_{\kk+\ell}^i(\sigma\shuffle\tau)=\partial_\kk^i(\sigma)\shuffle \tau$;
 
 \vspace{0.2em}
 
 -- $(ii)$ For all $i\in[\kkp,\kk+\ell]$, we have $\partial_{\kk+\ell}^i(\sigma\shuffle\tau)=\sigma\shuffle\partial_\ell^{i-\kk}(\tau)$.
\end{lem}

\begin{proof}
Let $\sigma$ and $\tau$ be two permutations of $\SSym\kk$ and $\SSym\ell$ and $\uu$, $\vv$ be  their respective words.
Let $\ii$ be an integer of $[1,\kk]$. 
Then, there exists a unique $\jj$ such that $\uu_\jj=\pm\ii$.
%As for the definition of $\sign\ii$ we use the conventions $\uu_0=0$ and $\uu_\kkp=-\infty$.
By definition of~$\del\ii$, we have $\del\ii(\uu)=\dec{\uu_1}\ii\,...\,\dec{\uu_\jjo}\ii\,\dec{\uu_\jjp}\ii\,...\,\dec{\uu_\kk}\ii$.
Let 
\[
 \YY=\{y_1<...<y_\kko\}
\]
be an element of $\SH\kko\ell$.
Writing $\uu'_m$ for $\dec{\uu_m}\ii$, there exists $\kk$ words $\vv_0,...,\vv_\kko$
such that $\vv_0\,...\,\vv_\kko=\vv$, $\ell(\vv_\jj)=y_\jjp-y_\jj-1$ for $\jj\in[1,k-2]$ and 
\begin{equation}
\label{E:Eq}
\del\ii(\uu)\shuffle^\YY\vv=\shift{\vv_0}\kko\,\uu'_1\,...\,\shift{\vv_{\jj-2}}\kko\,\uu'_\jjo\,\shift{\vv_\jjo}\kko\,\uu'_\jjp\,\shift{\vv_\jj}\kko\,...\,\uu'_\kk\,\shift{\vv_\kko}\kko.
\end{equation}

We now express $\vv_\jjo$ as the word $\alpha_1\,...\,\alpha_m$, with $m=y_\jjp-y_\jj-1$.
For $a\in[0,m]$, we define $\kkp$ words $\ww_0^a,...,\ww_\kk^a$ by 
\begin{equation*}
 \ww_p^a=\begin{cases}
           \vv_p&\text{for $p\leq \jj-2$,}\\
           \alpha_1\,...\,\alpha_a&\text{for $p=\jjo$,}\\
           \alpha_{a+1}\,...\,\alpha_m&\text{for $p=\jj$,}\\
           \vv_{p-1}&\text{for $p\geq\jjp$}.
          \end{cases}
\end{equation*}
Then $\vv$ is equal to $\ww_0^a\ ...\ \ww_\jjo^a\, \ww_\jj^a \, ...\, \ww_\kk^a$.
We define $\YY_a$, the refinement of $\YY$,  by  
\begin{equation*}
Y_a=\{y_1<...<y_{\jjo}<y_{\jjo}+a+1<y_{\jj}+1<...<y_\kko\}. 
\end{equation*}
Note that $Y_a$ is an element of $\SH\kk\ell$ for all values of $a\in[0,m]$.
The shuffle product of $\uu$ and $\vv$ relatively to $Y_a$ is 
\begin{equation*}
\uu\shuffle^{Y_a}\vv=\shift{\ww_0^a}\kk\,\uu_1\,...\,\shift{\ww_{j-2}^a}\kk\,\uu_\jjo\,\shift{\ww_{j-1}^a}\kk\,\uu_\jj\,\shift{\ww_{\jj}^a}\kk\,\uu_\jjp\,\shift{\ww_{\jjp}^a}\kk\,...\,\uu_\kk\,\shift{\ww_\kk^a}\kk,
\end{equation*}
Applying $\del\ii$ to the previous relation gives that $\del\ii\left(\uu\shuffle^{Y_a}\vv\right)$ is equal to
\begin{align*}
\shift{\ww_0^a}\kko\,\uu'_1...\,\shift{\ww_{j-2}^a}\kko\,\uu'_\jjo\,\cdot\,\shift{\ww_{j-1}^a}\kko\,\shift{\ww_{\jj}^a}\kko\,\cdot\,\uu'_\jjp\,\shift{\ww_{\jjp}^a}\kko\,...\uu'_\kk\,\shift{\ww_\kk^a}\kko,
\end{align*}
which, by definition of the words $\ww_p^a$, is  exactly the expression of $\del\ii(\uu)\shuffle^\YY\vv$ given in \eqref{E:Eq}.
We then obtain
\begin{equation*}
\sum_{a=0}^m \sign\ii\left(\uu\shuffle^{Y_a}\vv\right) \del\ii\left(\uu\shuffle^{Y_a}\vv\right)=\left(\sum_{a=0}^m \sign\ii\left(\uu\shuffle^{Y_a}\vv\right)\right)\del\ii(u)\shuffle^\YY\vv.
\end{equation*}
By definition of $\sign\ii$ and $\eps$ together with 
the conventions $\alpha_0=\uu_\jjo$, $\alpha_{m+1}=\uu_\jjp$, 
and the conventions $\uu_0=0$, $\uu_\kkp=-\infty$ used in definition of $\sign{}$, we obtain
\begin{align*}
 \sum_{a=0}^m\sign\ii\left(u\shuffle^{Y_a}v\right)&=
 \sum_{a=0}^\kk\eps(\alpha_{a},\uu_\jj,\alpha_{a+1})
 =\frac12\sum_{a=0}^\kk\eps(\alpha_a,u_j)+\eps(u_j,\alpha_{a+1})\\
 &=\frac12\sum_{a=0}^\kk\big(\eps(\alpha_a,u_j)-\eps(\alpha_{a+1},u_j)\big)
 =\eps(\alpha_0,u_j,\alpha_{m+1}),
 \end{align*}
and the latter is equal to $\eps(\uu_\jjo,\uu_\jj,\uu_\jjp)$ which is $\sign\ii(\uu)$.
We have then proved 
\begin{equation*}
\sum_{a=0}^m \sign\ii\left(\uu\shuffle^{Y_a}\vv\right) \del\ii\left(\uu\shuffle^{Y_a}\vv\right)=\sign\ii(\uu)\del\ii(\uu)\shuffle^Y\vv.
\end{equation*}
From the relation $\SH\kk\ell=\{\YY_a\,|\, \YY\in\SH\kko\ell\,\text{and}\, a\in[0,y_\jjp-y_\jj-1]\}$,
we get
\begin{align*}
 \partial^i_\kk(\sigma)\shuffle\tau&=\sum_{Y\in\SH\kko\ell}\sign\ii(\uu)\ \del\ii(\uu)\shuffle^\YY\vv\\
 &=\sum_{Y\in\SH\kko\ell}\sum_{a=0}^{y_\jjp-y_\jj-1} \sign\ii\left(\uu\shuffle^{Y_a}\vv\right) \del\ii\left(\uu\shuffle^{Y_a}\vv\right)\\
 &=\sum_{X\in\SH\kk\ell} \sign\ii\left(\uu\shuffle^X\vv\right) \del\ii\left(\uu\shuffle^X\vv\right)\\
 &=\partial_{\kk+\ell}^\ii\left(\sigma\shuffle\tau\right).
\end{align*}
We prove $(ii)$ with a similar argument, exchanging the role of $\uu$ and $\vv$.
\end{proof}

From the compatibility of $\partial$ whith the shuffle product $\shuffle$, we determine the image of $P_n$ under the derivation~$\partial$.

\begin{lem}
\label{L:DerivI}
 For all $\nn\geq 1$, we have $\partial(I_n)=(n-1) I_\nno$, $\partial(J_n)=(n-2)J_\nno$ and $\partial(P_\nn)=(n-2)P_{n-1}$, 
 with the conventions $I_0=J_0=P_0=\emptyset$.
\end{lem}

\begin{proof}
For $\nn\geq 1$, we have 
\begin{equation*}
 \partial(I_n)=\sum_{\ii=1}^n \partial_\ii(I_n)=\sum_{\ii=1}^\nn \sign\ii(I_\nn)I_\nno
\end{equation*}
By definition of $\sign{}$, we have $\sign1(I_\nn)=...=\sign\nno(I_\nn)=1$ and $\sign\nn(I_\nn)=0$.
These imply $\partial(I_\nn)=(\nn-1)I_\nno$.
Similary, since $\sign1(J_\nn)=-1$, $\sign\kk(J_\nn)$ is~$1$ for $\kk\in[2,\nno]$
and $\sign\nn(J_\nn)=0$, we obtain $\del\ii(J_\nn)=J_\nno$.
This implies $\partial(J_\nn)=(\nn-2)J_\nno$ for $\nn\geq 1$.
Let us now prove $\partial(P_\nn)=(n-2)P_{n-1}$.

By convention, we have $\partial(I_0)=\partial(J_0)=0$.
Using Lemma~\ref{L:Pshuffle} and the compatibility of $\partial$ and $\shuffle$ given 
in Lemma~\ref{L:CompPart}, we obtain
\begin{align*}
 \partial(P_n)&=\partial\left(\sum_{k=0}^\nn J_\kk\shuffle I_{\nn-\kk}\right)\\
 &=\sum_{k=0}^\nn \partial(J_\kk)\shuffle I_{\nn-\kk} + \sum_{k=0}^\nn J_\kk\shuffle\partial(I_{\nn-\kk})\\
 &=\sum_{k=1}^\nn (\kk-2) J_\kko\shuffle I_{\nn-\kk} + \sum_{k=0}^\nn (n-k-1)J_\kk\shuffle I_{\nn-\kk-1}\\
 &=\sum_{k=0}^\nno (\kk-1) J_\kk\shuffle I_{\nn-1-\kk}  + \sum_{k=0}^\nn (n-k-1)J_\kk\shuffle I_{\nn-1-\kk}\\
&=(\nn-2) \sum_{k=0}^\nno J_\kk\shuffle I_{\nn-1-\kk} = (\nn-2) P_\nno.
 \end{align*}
\end{proof}

The proof of the surjectivity of $\partial_\nn$ given in Proposition~\ref{P:Surj} uses a triangular 
argument that we  illustrate on an example:

\begin{exam}
 Let $\sigma=\sigma_1$ be the permutation $(2,-1,4,5,-3)$ of $\SSym5$.
 We look after the maximal sequence of the form $\kk\,...\,5$ or $-\kk\,...\,-5$ in the word~$w(\sigma)$.
 In our example, this sequence is $4,5$.
 We define $\tau_1$ to be the permutation $(2,-1,4,5,6,-3)$ obtained form $\sigma_1$ by replacing  $4,5$  by $4,5,6$.
 A direct computation gives $\partial_6(\tau_1)=2\sigma_1-\sigma_2$ with $\sigma_2=(2,-1,3,4,5)$.
 Hence, we obtain 
 \begin{equation}
  \label{E:DerivSurj}
  \sigma_1=\partial_6\left(\frac12\tau_1\right)+\frac12\sigma_2.
 \end{equation}
The maximal sequence of the desired form in $\sigma_2$ is $3,4,5$, which is longer than this of~$\sigma_1$.
We then define $\tau_2$ to be $(2,-1,3,4,5,6)$ and we compute $\partial_6(\tau_2)=3\sigma_2$.
Hence $\sigma_2$ is equal to $\partial_6\left(\frac13\tau_2\right)$ and, eventually,
substituting this to \eqref{E:DerivSurj}, we obtain
\[
 \sigma_1=\partial_6\left(\frac12\tau_1\right)+\partial_6\left(\frac16\tau_2\right)=\partial_6\left(\frac12(2,-1,4,5,6,-3)+\frac16(2,-1,3,4,5,6)\right).
\]

\end{exam}

\begin{prop}
\label{P:Surj}
For all $n\in\NN$, the map $\partial_\nnp:\QQ\SSym\nnp\to\QQ\SSym\nn$ is surjective.
\end{prop}

\begin{proof}
 Let $\sigma$ be a permutation of $\SSym\nn$. We denote the word $\ww(\sigma)$ by $\uu$.
 We have two cases, depending if $\nn$ or $-\nn$ appears in $\uu$.
 
\textbf{Case $\nn$ appears in $\uu$}:
we define $\ii(\sigma)$ to be the minimal integer such that $u$ can be written
as $\vv\cdot [\ii\  ...\  \nn] \cdot \ww$. 
We use an induction on $\ii(\sigma)$.
If $\ii(\sigma)$ is $1$ then $\sigma=I_\nn$. 
As Lemma~\ref{L:DerivI} gives 
$$\partial_\nnp(I_\nnp)=\nn I_\nn,$$
 we obtain $\sigma=\partial_\nnp(\frac{1}\nn I_\nnp)$.
Assume now $\ii=\ii(\sigma)>1$. 
Let $\uu'$  be the word $\vv\cdot[\ii\ ...\ \nnp] \cdot\ww$. 
Since each letter of~$\vv$ and $\ww$ are smaller than $\iio$, the word $\uu'$ belongs to $\WW\nnp$.
We denote by $\tau$ the permutation of $\SSym\nnp$ attached to $\uu'$.
For $\jj\in\{\ii,...,\nn\}$ we have $\del\jj(\uu')=\uu$ and $\sign\jj(\uu')=1$. 
As the first letter of $\ww$ is smaller than $\nnp$, we obtain $\sign\nnp(\uu')=0$, and so
\[
\sum_{j=\ii}^\nnp \partial_\nnp^\jj(\tau)=(\nn-\ii+1)\sigma,
\]
with $\nn-\ii+1\not=0$, since $\ii\leq\nn$.
Let $\jj$ be in $\{1,...,\iio\}$. 
Since the letter $\pm\jj$ appears only in $\vv$ or in $\ww$ and $|\jj|<\ii$,
we have 
$$\del\jj(\uu')=\dec\vv\jj\cdot [\iio\ ...\ \nn] \cdot \dec\ww\jj.$$
We then obtain that $\partial_\nnp(\tau)$ is the sum of $(\nn-\ii+1)\sigma$ and a linear 
combination of permutations $\alpha_1,...,\alpha_\kk$ of $\SSym\nn$ satisfying $\ii(\alpha_\jj)=\iio<\ii=\ii(\sigma)$. 
By the induction hypothesis, the $\alpha_\jj$'s belong to $\Ima(\partial_\nnp)$, which implies $\sigma\in\Ima(\partial_\nnp)$.

\textbf{Case  $\nn$ does not appear in $\uu$}: hence, $-\nn$ appears in $\uu$.
We now define $\ii(\sigma)$ to be the minimal integer such that $u$ can be written
as $\vv\cdot [-\ii\  ...\  -\nn] \cdot \ww$. 
We use also an induction on $\ii(\sigma)$.
For $\kk\in\NN$, we denote by $J_\kk$ the permutation $(-1,...,-\kk)$ of~$\SSym\kk$.
If $\ii(\sigma)=1$,  then $\sigma=J_\nn$.
By Lemma~\ref{L:DerivI} we have
$$\partial_\nnp(J_\nnp)=-(\nnp) J_\nn,$$
we obtain $\sigma=\partial_\nnp(-\frac{1}\nnp J_\nnp)$.
Assume now $\ii=\ii(\sigma)>1$. 
We denote by $\uu'$ the word $\vv\cdot[-\ii\ ...\ -(\nnp)] \cdot \ww$ of $\WW\nnp$ and by $\tau$ the corresponding permutation of~$\SSym\nnp$.
For $\jj<\ii$, we have 
$$\del\jj(\uu')=\dec\vv\jj\cdot[-(\iio)\ ... \ -\nn]\cdot\dec\ww\jj.$$ 
Hence $\alpha=\sum_{\jj=1}^\iio \partial_\nnp^\jj(\tau)$ is a linear combination of permutations $\alpha_1,....,\alpha_\kk\in\SSym\nn$ 
satisfying $\ii(\alpha_\jj)<\ii=\ii(\sigma)$ which, by induction hypothesis, implies $\alpha\in\Ima(\partial_\nnp)$.
It remains to establish that $\beta=\sum_{\jj=\ii}^\nnp\partial_\nnp^\jj(\tau)$ is a multiple of $\sigma$.
For $\jj\in\{\ii,...,\nn\}$, we have $\del\jj(\uu')=\uu$ and $\sign\jj(\uu')=-1$. 
If $\ww$  empty, then $\sign\nnp(\tau)=-1$ and $\del\nnp(\uu')=\uu$.
So, in this case, $\beta$ is equal to $-(\nn+2-\ii)\sigma$ with $\nn+2-\ii\not=0$, since $\ii\leq\nn$.
If $\ww$ is not empty, then its first letter is greater than $-(\nnp)$, 
implying $\sign\nnp(\tau)=0$. Then, $\beta=-(\nn+1-\ii) \sigma$ with $\nn+1-\ii\not=0$, since $\ii\leq \nn$.
In all cases, we obtain that $\sigma$ belongs to the image of $\partial_\nnp$. 
\end{proof}

\begin{coro}
\label{C:Deriv}
 The map $\partial$ is a surjective derivation of $(\BFQSym,\shuffle)$.
\end{coro}

\begin{proof}
 Let $\sigma$ and $\tau$ be two signed permutations of $\SSym\kk$ and $\SSym\ell$.
 By definition of $\partial$, we have the relation
 \begin{equation*}
  \partial(\sigma\shuffle\tau)
  =\sum_{\ii=1}^\nn\partial_{\kk+\ell}^i(\sigma\shuffle\tau)
  =\sum_{\ii=1}^\kk\partial_{\kk+\ell}^\ii(\sigma\shuffle\tau)
  +\sum_{\ii=\kkp}^{\kk+\ell}\partial_{\kk+\ell}^\ii(\sigma\shuffle\tau).
 \end{equation*}
Thus, by Lemma~\ref{L:CompPart}, we obtain
\begin{equation*}
 \partial(\sigma\shuffle\tau)=
 \sum_{\ii=1}^\kk\partial_\kk^\ii(\sigma)\shuffle\tau+
 \sum_{\ii=1}^\ell\sigma\shuffle\partial_\ell^\ii(\tau),
\end{equation*}
and so $\partial(\sigma\shuffle\tau)=\partial(\sigma)\shuffle\tau+\sigma\shuffle\partial(\tau)$.
The surjectivity statement is given by Proposition~\ref{P:Surj}.
\end{proof}

\subsection{Commutation of $\partial$ and $\Phi$.}

We shall now prove that $\partial$ and $\Phi$ commutes.
We start with two intermediate results.

\begin{lem}
\label{L:Thrm1}
 For all $\sigma\in\SSym\nn$ and $\ii\in[1,\nn]$, we have
 \begin{equation*}
  \Des{\del{|\sigma(\ii)|}(\sigma)}=\begin{cases}
                         D_i&\text{for $\sigma(i-1)<\sigma(i+1)$;}\\
                         D_i\cup\{i-1\}&\text{for $\sigma(i-1)>\sigma(i+1)$.}
                        \end{cases}
 \end{equation*}
 where $D_\ii=\Des\sigma\cap[0,\ii-2]\cup\left\{d-1\,|\,d\in\Des\sigma\cap[\iip,\nn]\right\}$ and the convention $\sigma(0)=0$.

\end{lem}

\begin{proof}
 Let $\uu$ be the word of a permutation $\sigma\in\SSym\nn$ and $\ii$ in $[1,\nn]$. 
 We denote by~$\jj$ the unique positive integer such that $\sigma(\jj)=|\ii|$ holds, 
 by $\vv$ the word $\del\jj(\uu)$ 
 and by $\tau$ the permutation $\perm(\vv)$.
 The word $v=v_1\,...\,v_\nno$ is then defined by
 \begin{equation*}
  v_\kk=\begin{cases}
       \dec{u_\kk}\ii&\text{for $\kk\leq \iio$,}\\
       \dec{u_\kkp}\ii&\text{for $\kk\geq \ii$.}
      \end{cases}
 \end{equation*}
 where $u_k$ and $v_k$ are the $k$-th letter of $u$ and $v$ respectively.
For $\kk\in[0,\nno]$,
we have $\dec{u_\kk}\ii>\dec{u_\kkp}\ii$ if, and only if, $u_\kk>u_\kkp$ holds 
(always with the convention $\uu_0=0$).
Hence, $\kk\in[0,\ii-2]$ is a descent of $\tau$ if, and only if, $\kk$ is a descent of~$\sigma$.
Similarly, $\kk$ in $[\ii,\nn-2]$ is a descent of $\tau$ if, and only if, $\kkp$ is a descent of $\sigma$.
Considering the set $D_\ii$ defined in the proposition, we have 
\begin{equation*}
 \Des\tau\cap\left([0,\nn-2]\setminus\{\iio\}\right)=D_\ii.
\end{equation*}
We cannot determine if $\iio$ is a descent of $\tau$ from $\Des\sigma$.
We only remark that the integer $\iio$ is a descent of $\tau$ 
if, and only if, $\vv_\iio>\vv_i$, hence if, and only if, we have $\uu_\iio>\uu_\iip$, as expected.
\end{proof}

\begin{lem}
\label{L:DesDeriv}
 Let $\sigma$ be a permutation of $\SSym\nn$ and $\{d_1<...<d_\kk\}$ be the set of its non-zero descents.
 For $i$ in $[1,\kk]$, we have 
 \begin{equation}
 \label{L:Thrm2:E:Term}
\Des{\sum_{e=d_\ii+1}^{d_\iip} \partial_{|\sigma(e)|}(\sigma)}=(d_\iip-d_\ii-2)\dec{\Des\sigma}{d_\iip}%\{d_1,...,d_\ii,d_\iip-1,...,d_\kk-1\}  
 \end{equation}
  with the convention $d_\kkp=\nn$. 
 Moreover,for $d_1>0$, \ie, $0\not\in\Des\sigma$, we have
 \begin{equation}
 \label{L:Thrm2:E:Term2}
 \Des{\sum_{e=1}^{d_1} \partial_{|\sigma(e)|}(\sigma)}=
  \begin{cases}
  (d_1-1)\dec{\Des\sigma}{d_1}&\text{if $0\not\in\Des\sigma$,}\\
  (d_1-2)\dec{\Des\sigma}{d_1}&\text{if $0\in\Des\sigma$.}\\
  \end{cases}
 \end{equation}
\end{lem}

\begin{proof}
 Let $\sigma$ be a permutation of $\SSym\nn$ and $\{d_1<...<d_\kk\}$ the set of its positive descents.
Let $\ii$ be an integer in $[1,\kko]$.
We start proving \eqref{L:Thrm2:E:Term} using three  subcases 

 \textbf{Case  $d_\iip>d_\ii+2$}: 
we have 
 \begin{equation*}
\sigma(d_\ii)>\sigma(d_\ii+1)<...<\sigma(d_\iip-1)<\sigma(d_\iip)>\sigma(d_\iip+1)
 \end{equation*}
By definition of $\sign{}$, the terms $\partial_{|\sigma(d_\ii+1)|}(\sigma)$ and 
$\partial_{|\sigma(d_\iip)|}(\sigma)$ are equal to $0$.
Let~$e$ be an element of $[d_\ii+2,d_\iip-1]$.
Then $\sign{|\sigma(e)|}(\sigma)$ is $1$.
By Lemma~\ref{L:Thrm1}, the descent set of $\del{|\sigma(e)|}$ is $\dec{\Des\sigma}{d_\iip}$,
 since $\sigma(e-1)<\sigma(e+1)$.
We conclude this case by remarking that the cardinality of $[d_\ii+2,d_\iip-1]$ is 
exactly $d_\iip-d_\ii-2$.

\textbf{Case  $d_\iip=d_\ii+2$}: 
we have 
\begin{equation*}
 \sigma(d_\ii)>\sigma(d_\ii+1)<\sigma(d_\iip)>\sigma(d_\iip+1).
\end{equation*}
As for $e\in[d_\ii+1,d_\iip]$, we have $\sign{|\sigma(e)|}(\sigma)=0$, 
the left hand side of~\eqref{L:Thrm2:E:Term} is $0$.

\textbf{Case  $d_\iip=d_\ii+1$}: 
we have $\sigma(d_\ii)>\sigma(d_\iip)>\sigma(d_\iip+1)$.
In this case, $\sign{|\sigma(d_\iip)|}$ is $-1$.
By Lemma~\ref{L:Thrm1}, the positive descents of $\del{|\sigma(d_\ii+1)|}(\sigma)$ are 
\begin{equation*}
\{d_1,...,d_\iio,d_\iip-1,...,d_\kk-1\}\cup\{d_\ii\}=\{d_1,...,d_\ii,d_\iip-1,...,d_\kk-1\}
\end{equation*}
since $\sigma(d_\ii)>\sigma(d_\ii+2)$ holds. 
We conclude by remarking  that $d_\iip-\dd_\ii-2=-1$ occurs in this case.

Relation~\eqref{L:Thrm2:E:Term2} is proved similarly, with a particular attention on $0$.
\end{proof}

\begin{thrm}
\label{T:Commute}
The endomorphisms $\Phi$ and $\partial$ commute.
\end{thrm}

\begin{proof}
 Let $\sigma$ be a permutation of $\SSym\nn$. 
 Let us denote by $\{d_1<...<d_\ell\}$ the set of non-zero descents of $\sigma$.
 For $\ii\in[1,\ell]$ we denote by $\kk_\ii$ the integer $\dd_\iip-\dd_\ii$, 
with the convention $\dd_0=0$ and $\dd_{\ell+1}=\nn$.
For $\kk\in\NN$, we define $X_{\kk}$ and $x_{\kk}$ by 
\begin{equation*}
 X_{\kk}=\begin{cases} I_{\kk} & \text{for $0\not\in \Des{\sigma}$,} \\ P_{\kk} &\text{for $0\in \Des{\sigma}$;}\end{cases}\quad\text{and}
 \quad x_{\kk}=\begin{cases} \kk-1 & \text{for $0\not\in \Des{\sigma}$,}\\ \kk-2 & \text{for $0\in \Des{\sigma}$}.\end{cases}
\end{equation*}

 By Proposition~\ref{P:PhiTilde}, we have 
 $\Phi(\sigma)=X_{\kk_1}\shuffle P_{\kk_2}\shuffle ... \shuffle P_{\kk_\ell}$.
 Since $\partial$ is a derivation, by Corollary~\ref{C:Deriv}, the previous relation gives
 \begin{align*}
  (\partial\circ\Phi)(\sigma)=&\partial(X_{k_1})\shuffle P_{k_2} \shuffle ... \shuffle P_{k_\ell}\\
  &+\sum_{i=2}^\ell X_{k_1}\shuffle ... \shuffle P_{\kk_\iio} \shuffle \partial( P_{k_\ii}) \shuffle P_{k_\iip} ...\shuffle P_{k_\ell}.
 \end{align*}
 and so, using Lemma~\ref{L:DerivI}, we obtain
 \begin{align*}
  (\partial\circ\Phi)(\sigma)&=x_{k_1}X_{k_1-1}\shuffle P_{k_2}\shuffle\,...\,\shuffle P_{\kk_\ell}\\
  &+\sum_{i=2}^\ell (k_i-2)X_{k_1}\shuffle P_{k_2}\shuffle\, ... \,
  \shuffle P_{k_\iio} \shuffle P_{\kk_\ii-1} \shuffle P_{\kk_\iip}\shuffle
  \, ... \, \shuffle P_{\kk_\ell}.
 \end{align*}
In other hand, by Lemma~\ref{L:DesDeriv}, we have 
\begin{equation*}
\Des{\partial(\sigma)}=x_{\kk_1} \dec{\Des{\sigma}}{d_1} + \sum_{i=2}^\ell(\kk_\ii-2) \dec{\Des{\sigma}}{d_\ii}
\end{equation*}
By Proposition~\ref{P:PhiTilde} we obtain  
$$
\widetilde{\Phi}_\nn(\dec{\Des\sigma}{d_1})=X_{k_1-1}\shuffle P_{k_2}\shuffle...\shuffle P_{\kk_\ell},
$$
and for $\ii$ in $[2,\nn]$ we have  
$$\widetilde{\Phi}_\nn(\dec{\Des\sigma}{d_\ii})=
X_{k_1}\shuffle P_{k_2}\shuffle\, ...\,
\shuffle P_{k_\iio} \shuffle P_{\kk_\ii-1} \shuffle P_{\kk_\iip}\shuffle \, ... \,\shuffle P_{\kk_\ell}
$$
Since $(\Phi\circ\partial)(\sigma)=(\widetilde{\Phi}_\nn(\Des{\partial(\sigma)}))$, we have established $(\Phi\circ \partial)(\sigma)=(\partial\circ \Phi)(\sigma)$.
\end{proof}

We can now prove the main theorem.

\begin{proof}[Proof of Theorem~\ref{T:Main}]
Let $\nn$ be an integer. 
By Corollary~\ref{C:Deriv}, the map $\partial$ is a surjective derivation of $\QQ\SSym{}$,
which, by Theorem~\ref{T:Commute}, commutes with $\Phi$.
Proposition~\ref{P:Div} guarantees that the characteristic polynomial of $\Phi_\nn$ divides the one of $\Phi_\nnp$.
Since the characteristic polynomial of $\Phi_\nn$ is the one of $\Adj_n^B$, we have established the 
expected divisibility result.
\end{proof}

\section{Other types}

In this section, we discuss about the becoming of the divisibility result for other infinite Coxeter families, and we 
describe the combinatorics of normal sequences of braids for some exceptionnal types.

Let $\Gamma$ be a finite connected Coxeter graph. 
From a computational point of view the matrix $\Adj_\Gamma$ is too huge, as its size is exactly the number of elements in $W_\Gamma$, whose growth in an exponential in $n$ for the family $A_n, B_n$ and~$D_n$.
 
The definition of the descent set given in Definition~\ref{D:Des} has a counterpart in~$W_\Gamma$ for every Coxeter graph $\Gamma$ (the reader can consult~\cite{bjorner2006ComCoxgro} for more details on the subject).

\begin{defi}
For $\Gamma$ a Coxeter graph we define a square matrix $\Adj'_\Gamma=(a'_{I,J})$  indexed by the subset of vertices of $\Gamma$ by:
\[
a'_{I,J}=\card\{w\in W_\Gamma\,|\, \Des{w\inv}=I \ \text{and}\ J\subseteq \Des{w}\}
\]
\end{defi}

For $\Gamma$ a graph of the family $A_n,B_n$ and $D_n$,
 the size of $\Adj'_\Gamma$ is $2^n$, which is smaller than $n!,2^nn!$ and $2^{n-1}n!$ respectively.

For any subset $J$ of $\Gamma$, we denote by $b_\Gamma^d(J)$, the numbers of positive braids of $B^+(W_\Gamma)$ whose Garside normal form is $(w_1,...,w_d)$ with $\Des{w_d}\subset J$. An immediate adaptation of Lemma~2.12 of \cite{dehornoy2007Comnorseqbra} gives:

\begin{lem}
\label{L:AdjP}
For $\Gamma$ a finite connected Coxeter graph, there exists an integer $k$ such that the characteristic polynomial $\chi_\Gamma(x)$ of $\Adj_\Gamma$ is equal to $x^k \chi'_\Gamma(x)$ where $\chi'_\Gamma(x)$ is the one of $\Adj'_\Gamma$.
Moreover, for  $d\geq1$ and $J\subset \Gamma$, we have 
\[
b_\Gamma^d(J)={}^tY\, (\Adj_\Gamma')^{d-1} J\quad \text{where}\quad Y_I=\begin{cases}0&\text{if $I=\emptyset$,}\\1&\text{otherwise.}\end{cases} 
\]
\end{lem}

In order to determine the numbers $b_{\Gamma}^d$ of braids of $B^+(W_\Gamma)$ whose Garside length is $d$ form $\Adj'_\Gamma$, we use an inclusion exclusion principle.

\begin{coro}
\label{C:AdjP}
For $\Gamma$ a finite connected Coxeter graph and $d\geq1$, we have:
\[
b_\Gamma^d={}^tY\, (\Adj_\Gamma')^{d-1} Z \quad \text{where}\quad Z_I=\begin{cases}0&\text{if $I=\emptyset$,}\\(-1)^{\card(I)+1}&\text{otherwise.}\end{cases}
\]
and $Y$ as in Lemma~\ref{L:AdjP}.
\end{coro}

\subsection{Braids of type $D$.}

For $n\geq 4$, the Coxeter graph of type $D$ and rank $n$ is
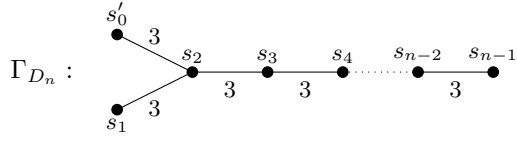
\[
\begin{tikzpicture}
\draw(0,0) node {$\Gamma_{D_n}:$};
\draw(1.5,0.25) node[above]{\small$3$};
\draw(1.5,-0.25) node[below]{\small$3$};
\draw(2.5,0) node[below]{\small$3$};
\draw(3.5,0) node[below]{\small$3$};
\draw(5.5,0) node[below]{\small$3$};
\draw(1,0.5) -- (2,0) -- (1,-0.5);
\draw(2,0) -- (4,0);
\draw(4,0) -- (5,0) [dotted];
\draw(5,0) -- (6,0);
\filldraw(1,0.5) circle (2pt) node[above]{\small$s'_0$};
\filldraw(1,-0.5) circle (2pt) node[below]{\small$s_1$};
\filldraw(2,0) circle (2pt) node[above]{\small$s_2$};
\filldraw(3,0) circle (2pt) node[above]{\small$s_3$};
\filldraw(4,0) circle (2pt) node[above]{\small$s_4$};
\filldraw(5,0) circle (2pt) node[above]{\small$s_{n-2}$};
\filldraw(6,0) circle (2pt) node[above]{\small$s_{n-1}$};
\end{tikzpicture},
\]
and the associated  Coxeter group is isomorphic to the subgroup of $\SSym\nnp$ consisting of all signed permutations with an even number of negative entries. 
Its generators are the signed permutations $s_i$ for $i\in[1,\nno]$, plus the signed permutation $s_0'=(-2,-1,3,...,n)$. 
We extend the family $D_n$ defined for $n\geq 4$ to include $D_1=A_1$, $D_2=A_1\times A_1$ and $D_3=A_3$.
Note that we usually only consider $n\geq 4$ in order to have a classification of irreducible Coxeter groups without redundancy.

Denoting by $\chi_{D_\nn}$ the characteristic polynomial of the adjacent matrix $\Adj_{D_n}$ of normal sequences of positive braid of type $D$ and rank $\nn$ we obtain:
\begin{align*}
 \chi_{D_1}(x)&=(x-1)^2\\
 \chi_{D_2}(x)&=(x-1)^4\\
 \chi_{D_3}(x)&=x^{19}\ (x-1)^2\ (x-2)\ (x^2-6x+3)\\
 \chi_{D_4}(x)&=x^{181} \ (x - 1)^6 \ (x^5 - 44x^4 + 402x^3 - 1084x^2 + 989 x - 360)\\
 \chi_{D_5}(x)&=x^{1906}\ (x-1)^2\ (x^{12}-302 x^{11}+17070 x^{10}-328426 x^9+3077800 x^8\\
 &\phantom{x^{1906}\ (x-1)^2\ (x^{12}}-16424030 x^7+4072794 x^6-113921686 x^5+154559655 x^4\\
 &\phantom{x^{1906}\ (x-1)^2\ (x^{12}}-132533636 x^3+68372600 x^2-18880000 x+2016000)
\end{align*}

As the reader can check, there is no hope to have a divisibility of  $\chi_{D_\nnp}$ by $\chi_{D_\nn}$ except for $n=1$.
The associated generating series are: 
\begin{align*}
  F_{D_2}(t)&=\frac{3-t}{(t-1)^2}\\
  F_{D_3}(t)&=\frac{-6t^3+15t^2-20t+23}{(t-1)(2t-1)(3t^2-6t-1)}\\
  F_{D_4}(t)&=\frac{-360t^5+1709t^4-2246t^3+852t^2+430t+191}{(t-1)(-1+44t-402t^2+1084t^3-989t^4+360t^5}
\end{align*}
which give the following values for the number of $D$-braids of rank $n$ and of Garside length $d$:
\[
\begin{array}{c|r|r|r|}
d&b_{D_2}(d)&b_{D_3}(d)&b_{D_4}(d)\\
\hline
0&1&23&191\\
1&3&187&9025\\
2&5&1169&321791\\
3&7&6697&10737025\\
4&9&37175&352664255\\
5&11&203971&11540908225
\end{array}
\]

\subsection{Braids of type I}
For $\nn\geq 2$, the Coxeter graph $I_n$ is 
 \[
\begin{tikzpicture}
\draw(0,0) node {$\Gamma_{I_n}:$};
\filldraw(1,0) circle (2pt) node[above]{$\small s$};
\filldraw(2.4,0) circle (2pt) node[above]{$\small t$};
\draw(1,0)--(2.4,0);
\draw(1.7,0) node[above]{$\small n$};

%\draw(0,0) node {$\Gamma_{D_n}:$};
%\draw(1.5,0.25) node[above]{\small$3$};
%\draw(1.5,-0.25) node[below]{\small$3$};
%\draw(2.5,0) node[below]{\small$3$};
%\draw(3.5,0) node[below]{\small$3$};
%\draw(5.5,0) node[below]{\small$3$};
%\draw(1,0.5) -- (2,0) -- (1,-0.5);
%\draw(2,0) -- (4,0);
%\draw(4,0) -- (5,0) [dotted];
%\draw(5,0) -- (6,0);
%\filldraw(1,0.5) circle (2pt) node[above]{\small$s'_0$};
%\filldraw(1,-0.5) circle (2pt) node[below]{\small$s_1$};
%\filldraw(2,0) circle (2pt) node[above]{\small$s_2$};
%\filldraw(3,0) circle (2pt) node[above]{\small$s_3$};
%\filldraw(4,0) circle (2pt) node[above]{\small$s_4$};
%\filldraw(5,0) circle (2pt) node[above]{\small$s_{n-2}$};
%\filldraw(6,0) circle (2pt) node[above]{\small$s_{n-1}$};
\end{tikzpicture},
\]
which gives the following presentation for the Coxeter group $W_{I_n}$: 
\begin{equation*}
W_{I_n}=\left<s,t\, \left| \begin{array}{c} 
			s^2=1, t^2=1\\
			\prodC(s,t;n)=\prodC(t,s;n) 
		   \end{array}\right.\right>.
\end{equation*}

 \begin{prop}
 \label{P:AdjI}
 For $\nn\geq 2$, we have 
 \begin{equation*}
 \Adj'_{I_n}=\begin{bmatrix}
 1&0&0&0\\
 n-1&b_n&a_n&0\\
 n-1&a_n&b_n&0\\
 n&1&1&1
 \end{bmatrix}
  \end{equation*}
 with $a_n=\lfloor\frac\nno2\rfloor$ and $b_n=\lfloor\frac\nn2\rfloor$
 \end{prop}

\begin{proof}
The elements of $W_{I_n}$ are $1$, $w_n=\prodC(s,t;n)=\prodC(t,s;n)$ and $\prodC(s,t;k)$ with $\prodC(t,s;k)$ for $k$ in $[1,n-1]$. 
For  $\kk$ in $[1,\nno]$, we have 
\begin{align*}
\prodC(s,t;k)\inv&=
\begin{cases}
\prodC(t,s;k)&\text{for $k$ even,}\\
\prodC(s,t;k)&\text{otherwise.}
\end{cases}
\\
\Des{\prodC(s,t;k)}&=
\begin{cases}
t&\text{for $k$ even,}\\
s&\text{otherwise.}
\end{cases}
\end{align*}

From the relation $\prodC(s,t;n)=\prodC(t,s;n)$ we have $w_n=\prodC(s,t;n)\inv=\prodC(s,t;n)$ and so
$\Des{w_n}=\{s,t\}$.
We organize the elements of $W_{I_n}\setminus\{1,w_n\}$ in $4$ sets: 
\begin{align*}
X_1&=\{\prodC(s,t;k)\ \text{for $\kk$ even}\},&X_2&=\{\prodC(s,t;k)\ \text{for $\kk$ odd}\},\\
X_3&=\{\prodC(t,s;k)\ \text{for $\kk$ even}\},&X_4&=\{\prodC(t,s;k)\ \text{for $\kk$ odd}\}.
\end{align*}
From the previous study of descents, we obtain
\[
\begin{array}{c|c|c|c|c|c|c|} 
\sigma\in&\{1\} &  X_1 &  X_2 &  X_3 & X_4 & \{w_n\}\\
\hline
\Des{\sigma} & \emptyset & \{t\} & \{s\} & \{s\} & \{t\} & \{s,t\}\\
\Des{\sigma\inv} & \emptyset & \{s\} & \{s\} & \{t\} & \{t\} & \{s,t\}\\
\end{array}
\]
Denoting by $a_n$ and $b_n$ the integers $\lfloor\frac\nno2\rfloor$ and $\lfloor\frac\nn2\rfloor$ respectively, we obtain that $\card(X_1)=\card(X_3)=a_n$ and $\card(X_2)=\card(X_4)=b_n$.
For $I,J$ subsets of~$\{s,t\}$ we define $A'_{I,J}$ to be the set $\{\sigma\in W_{I_n}\,|\,\Des{\sigma\inv}=I\ \text{and}\ J\subseteq\Des{w}\}$.
For all $K\subset\{s,t\}$ we have $A'_{\{s,t\},K}=\{w_n\}$. We have $A'_{\emptyset,\emptyset}=\{1\}$ and $A'_{\emptyset,K}=\emptyset$ for $K\not=\emptyset$. From the $X_i$'s we get
\begin{align*}
A'_{\{s\},\emptyset}&=X_1\sqcup X_2, \quad A'_{\{s\},\{s\}}=X_2, \quad A'_{\{s\},\{t\}}=X_1, \quad A'_{\{s\},\{s,t\}}=\emptyset,\\
A'_{\{t\},\emptyset}&=X_3\sqcup X_4, \quad A'_{\{t\},\{s\}}=X_3, \quad A'_{\{t\},\{t\}}=X_4, \quad A'_{\{t\},\{s,t\}}=\emptyset,
\end{align*}
Using the enumeration $\{\emptyset,\{s\},\{t\},\{s,t\}\}$ of subsets of $\{s,t\}$ together with the relation $a_n+b_n=n-1$ we obtain:
\begin{equation*}
\Adj'_{I_n}=\begin{bmatrix}
1&0&0&0\\
a_n+b_n&b_n&a_n&0\\
a_n+b_n&a_n&b_n&0\\
1&1&1&1
\end{bmatrix}
=
\begin{bmatrix}
1&0&0&0\\
\nno&b_n&a_n&0\\
\nno&a_n&b_n&0\\
1&1&1&1
\end{bmatrix}
\end{equation*}.
\end{proof}

\begin{coro}
The characteristic polynomial of $\Adj_{I_n}$ is 
\begin{equation*}
\chi_{I_n}(x)=\begin{cases}
x^{2n-4}(x-1)^3(x-n+1)&\text{if $x$ is even,}\\
x^{2n-3}(x-1)^2(x-n+1)&\text{otherwise.}
\end{cases}
\end{equation*}
and the generating series of normal sequence of $I_n$-braids is 
\begin{equation*}
F_{I_n}(t)=\frac{(n-1)t+1}{((n-1)t-1)(t-1)}.
\end{equation*}
\end{coro}

\begin{proof}
From the expression of $\Adj'_{I_n}$ given in Proposition~\ref{P:AdjI} we obtain 
\begin{align*}
\chi_{\Adj'_{I_n}}(x)&=(1-x)^2((b_n-x)^2-a_n^2)\\
&=(1-x)^2(b_n+a_n-x)(b_n-a_n-x)\\
&=(x-1)^2(x-(b_n+a_n))(x-(b_n-a_n))
\end{align*}
From the relations 
\begin{equation*}
a_n+b_n=n-1,\qquad b_n-a_n=\begin{cases} 1 &\text{if $n$ is even,}\\0&\text{otherwise.}\end{cases}  
\end{equation*}
we obtain
\begin{equation*}
\chi_{\Adj'_{I_n}}(x)=\begin{cases}
(x-1)^3(x-n+1)&\text{if $x$ is even,}\\
x(x-1)^2(x-n+1)&\text{otherwise,}
\end{cases}
\end{equation*}
Adding the missing powers of $x$ to obtain a degree of $2n$ we obtain the expected value for $\chi_{I_n}$.

For generating series results, Corollary~\ref{C:AdjP} gives 
\begin{equation*}
F_{I_n}(t)=\begin{bmatrix}0&1&1&1\end{bmatrix}\ (I_4-t\,\Adj'_{I_n})\inv\ \begin{bmatrix}0\\1\\1\\-1\end{bmatrix}.
\end{equation*}
By a direct computation (or a use of Sage~\cite{Sage} for example) we obtain
\begin{equation*}
F_{I_n}(t)=\frac{(n-1)t+1}{((n-1)t-1)(t-1)}.
\end{equation*}
\end{proof}

\subsection{Exceptional Coxeter groups}
Using $\Adj'_\Gamma$, we can study the combinatorics of normal sequence of braids of type $F_4,H_3,H_4,E_6$ and $E_7$.
The matrices $\Adj'_\Gamma$ were obtained using \texttt{Sage} \cite{Sage}, while the characteristic polynomials and generating series was obtained using the \texttt{C} library \texttt{flint} \cite{flint}.

The  group $W_{F_4}$ has $1152$ elements.
The characteristic polynomial of $\Adj_{F_4}$ is 
\begin{align*}
\chi_{F_4}(x)=&x^{1140}\,(x - 1)^3\,(x - 4)\, (x^2 - 25\,x + 10)\\
& (x^6 - 274\,x^5 + 9194\,x^4 - 77096\,x^3 + 250605\,x^2 - 324870\,x + 138600)
\end{align*}
and the generating series $F_{F_4}$ is given by
\begin{equation*}
F_{F_4}(t)=\frac{138600 \, t^{6} - 187350 \, t^{5} - 32055 \, t^{4} + 87970 \, t^{3} - 15504 \, t^{2} - 876 \, t - 1}{{\left(138600 \, t^{6} - 324870 \, t^{5} + 250605 \, t^{4} - 77096 \, t^{3} + 9194 \, t^{2} - 274 \, t + 1\right)} {\left(t - 1\right)}}.
\end{equation*}
The group $W_{H_3}$ has $120$ elements.
The characteristic polynomial of $\Adj_{H_3}$ is 
\begin{align*}
\chi_{H_3}(x)=x^{114}\,(x - 1)^2\, (x^4 - 42\,x^3 + 229\,x^2 - 244\,x + 72),
\end{align*}
and the generating series $F_{H_3}$ is given by
\begin{equation*}
F_{H_3}(t)=-\frac{72 \, t^{4} - 196 \, t^{3} + 77 \, t^{2} + 76 \, t + 1}{{\left(72 \, t^{4} - 244 \, t^{3} + 229 \, t^{2} - 42 \, t + 1\right)} {\left(t - 1\right)}}.
\end{equation*}
The group $W_{H_4}$ has $14400$ elements.
The characteristic polynomial of $\Adj_{H_4}$ is 
\begin{align*}
\chi_{H_4}(x)=x^{14390}\, (x - 1)^2 \, (x^8 &- 3436\,x^7 + 565470\,x^6 - 11284400\,x^5 + 81322353\,x^4\\
&- 246756500\,x^3 + 305430848\,x^2 - 157717504\,x + 27929088),
\end{align*}
and the generating series $F_{H_4}(t)=\frac{N_{H_4}(t)}{D_{H_4}(t)(t-1)}$ is given by
\begin{align*}
N_{H_4}(t)=27929088 \, t^{8}& - 147220480 \, t^{7} + 247258432 \, t^{6} - 138197780 \, t^{5} \\
&+ 465433 \, t^{4} + 10247814 \, t^{3} - 1205944 \, t^{2} - 10962 \, t - 1,\\ \\
D_{H_4}(t)=27929088 \, t^{8} &- 157717504 \, t^{7} + 305430848 \, t^{6} - 246756500 \, t^{5} \\
&+ 81322353 \, t^{4} - 11284400 \, t^{3} + 565470 \, t^{2} - 3436 \, t + 1.
\end{align*}

The group $W_{E_6}$ has $51840$ elements.
The characteristic polynomial of $\Adj_{E_6}$ is 
\begin{align*}
\chi_{E_6}(x)=&x^{51823}\, (x-1)^2\, \\
&(x^{15} - 5454\,x^{14} + 3391893\,x^{13} - 424089882\,x^{12} + 19590731031\,x^{11} \\
&\ - 417118001254\,x^{10} + 4673188683575\,x^9 - 29907005656510\,x^8 \\
&\ + 115900067128500\,x^7 - 282097630883500\,x^6 + 439789995997000\,x^5 \\
&\ - 441496921502000\,x^4 + 282303310340000\,x^3 - 110981554480000\,x^2 \\
&\ + 24563716800000\,x - 2328480000000),
\end{align*}
and the generating series $F_{E_6}(t)=\frac{N_{E_6}(t)}{D_{E_6}(t)(t-1)}$ is given by
\begin{align*}
N_{E_6}(t)=&2328480000000 \, t^{15} - 19422916800000 \, t^{14} + 59384818480000 \, t^{13} \\
&\ - 64287293380000 \, t^{12} - 64835775106000 \, t^{11} + 254118878161000 \, t^{10} \\
&\ - 284082015723500 \, t^{9} + 148526420487700 \, t^{8} - 32460183476310 \, t^{7} \\
&\ - 327255378405 \, t^{6} + 1042966224156 \, t^{5} - 93297805141 \, t^{4}\\
&\ + 479267710 \, t^{3} + 40099205 \, t^{2} + 46384 \, t + 1, \\ \\
D_{E_6}(t)=&2328480000000 \, t^{15} - 24563716800000 \, t^{14} + 110981554480000 \, t^{13} \\
&\ - 282303310340000 \, t^{12} + 441496921502000 \, t^{11} - 439789995997000 \, t^{10}\\
&\ + 282097630883500 \, t^{9} - 115900067128500 \, t^{8} + 29907005656510 \, t^{7} \\
&\ - 4673188683575 \, t^{6} + 417118001254 \, t^{5} - 19590731031 \, t^{4}\\
&\ + 424089882 \, t^{3} - 3391893 \, t^{2} + 5454 \, t - 1.
\end{align*}

The previous generating series gives the following values for  $b_W(d)$, the numbers of $W$-braids of Garside length $d$:

\begin{equation*}
\begin{array}{c|r|r|r|r|}
d&b_{F_4}(d)&b_{H_3}(d)&b_{H_4}(d)&b_{E_6}(d)\\
\hline
0&1&1&1&1\\
1&1151&119&14399&51839\\
2&322561&4923&50126401&319483603\\
3&77804927&179717&164094364799& 1567574732717\\
4&18441371521&6449741&535645654732801&7487770421878165\\
5& 4362177487103&230926603&1748252504973355199&35655729684940971035
\end{array}
\end{equation*}

The characteristic polynomial and the generating series for braids of type $E_7$ are available at \url{http://www.lmpa.univ-littoral.fr/~fromentin/combi.html}.

\bibliographystyle{plain}
\bibliography{biblio}
 \end{document}